\documentclass[oneside,reqno]{amsart}
\usepackage{amssymb}
\usepackage{graphicx}
\usepackage{amssymb}
\usepackage[active]{srcltx}
\usepackage{a4wide}
\usepackage{amsmath}
\usepackage{amssymb}
\usepackage{amstext}
\usepackage{cite}
\usepackage{epsfig}
\usepackage{enumerate}
\usepackage{color}
\usepackage{amssymb}
\usepackage{float}
\everymath{\displaystyle}
\newtheorem{theorem}{Theorem}[section]
\newtheorem{lemma}[theorem]{Lemma}
\usepackage[linkcolor=blue, urlcolor=blue, citecolor=blue,
colorlinks, bookmarks]{hyperref}
\theoremstyle{definition}
\newtheorem{definition}[theorem]{Definition}

\newtheorem{example}[theorem]{Example}

\newtheorem{note}[theorem]{Note}

\theoremstyle{remark}
\newtheorem{remark}[theorem]{Remark}

\newcommand{\R}{\mathbb{R}}

\newcommand{\be}{\begin{equation}}
\newcommand{\ee}{\end{equation}}
\numberwithin{equation}{section}

\usepackage{scalerel,stackengine}
\stackMath
\newcommand\reallywidehat[1]{%
\savestack{\tmpbox}{\stretchto{%
  \scaleto{%
    \scalerel*[\widthof{\ensuremath{#1}}]{\kern-.6pt\bigwedge\kern-.6pt}%
    {\rule[-\textheight/2]{1ex}{\textheight}}%WIDTH-LIMITED BIG WEDGE
  }{\textheight}% 
}{0.5ex}}%
\stackon[1pt]{#1}{\tmpbox}%
}

\begin{document}

\title[]{Box Dimension and Fractional Integrals of Multivariate Fractal Interpolation Functions}


\author{Vishal Agrawal}
\address{Department of Mathematical Sciences, IIT(BHU), Varanasi, India 221005}
\email{vishal.agrawal1992@gmail.com}
\author{Megha Pandey}
   % Address of record for the research reported here
\address{Department of Mathematical Sciences, IIT(BHU), Varanasi, India 221005}
    %Current address
\email{meghapandey1071996@gmail.com}
\author{Tanmoy Som}
\address{Department of Mathematical Sciences, IIT(BHU), Varanasi, India 221005}
\email{tsom.apm@iitbhu.ac.in}

%    General info
%\subjclass[2000]{Primary 54C40, 14E20; Secondary 46E25, 20C20}

%\date{January 1, 2001 and, in revised form, June 22, 2001.}


\begin{abstract}
In this article, we construct the multivariate fractal interpolation functions for a given data points and explore the existence of $\alpha$-fractal function corresponding to the multivariate continuous function defined on $[0,1]\times \cdots \times [0,1](q\text{-times})$. The parameters are selected such that the corresponding fractal version preserves some of the original function's properties, for instance if the given function is H\"older continuous, then the corresponding $\alpha$-fractal function is also H\"older continuous. Moreover, we explore the restriction of $\alpha$-fractal function on the co-ordinate axis. Furthermore, the box dimension and Hausdorff dimension of the graph of the multivariate $\alpha$-fractal function and its restriction are investigated. In the last section, we prove that mixed Riemann-Liouville fractional integral of fractal function satisfies a self-referential equation. 
\end{abstract}

\maketitle
\section{Introduction}
In \cite{MF1}, Barnsley used the Iterated Function System (IFS) to develop Fractal Interpolation Function (FIF). It is vital to select IFS adequately so that it can be used as an attractor for a graph of a continuous function referred to as FIF. Calculating the dimensions of fractal functions has been a major aspect of fractal geometry since its origin. Barnsley has established estimates for the Hausdorff dimension of an affine FIF in \cite{MF1} and Falconer \cite{Fal} has reported a related conclusion. 

Barnsley and associates \cite{MF1,MF2,MF4} determined the box dimension of the affine FIF classes. FIFs constructed by bilinear functions have been investigated in \cite{MF4}, a formula for the box dimension of FIFs from $\R^n$ to $\R^m$ can be seen in \cite{PM1}. Theories regarding the box dimension and fractional integral of FIF has been developed in the literature for instance, in \cite{RuanJAT2}, Ruan et.al., has studied the box dimension and fractional integral of linear FIFs. Wang et.al., \cite{WY} proposed the analytical properties of FIFs with variable parameters.
There is a special type of FIF that has been developed in the literature known as $\alpha$- fractal function \cite{Nav}. This type of FIFs is found to be very useful.
In \cite{BDD}, construction of recurrent bivariate Fractal Interpolation Surfaces (FIS) and estimation of their box dimension has been explored,  Chandra and Abbas \cite{SS} have studied fractional calculus of fractal surfaces developed by Ruan and Xu \cite{Ruan}. Agrawal and Som \cite{EPJST} have investigated the fractal dimension of the $\alpha$-fractal function on the Sierpi\'nski Gasket . Sahu and Priyadarshi have computed the box-counting dimension of graphs of fractal interpolation functions and harmonic functions on the Sierpi\'nski Gasket  in \cite{Sahu}. Verma et.al., have studied the Box dimension and Hausdorff dimension of the graph of the bivariate fractal function. Author of \cite{Maly} has calculated the exact values for the Minkowski dimension of the bivariate FIS and extended this construction to higher dimensions. Verma \cite{Verma2} has studied some important results of the attractor generated by a finite set of contractive infinitesimal similitudes.  In \cite{SSn}, box dimension of the graph of the mixed Weyl–Marchaud fractional derivative of a continuous function defined on $\mathbb{R}^2$ has been derived by Chandra and Abbas. Ruan, Xiao and Yang \cite{Ruann2} have extended the concept of recurrent FIFs such that iterated functions in a recurrent system do not have to be contractive with respect to the first variable. Furthermore, they have determined the box dimensions of all self-affine recurrent FIFs.
In \cite{Zhen}, Liang and Ruan have constructed the recurrent FISs. Furthermore, they have proposed recurrent bilinear FISs. Then, they have calculated the box dimension of bilinear recurrent FISs under certain constraints. In \cite{MR}, Roychowdhury has estimated the Hausdorff and upper box dimension of hyperbolic recurrent sets. Recently,  Verma and Sahu have initiated a systematic study of fractal dimension of functions defined on the Sierpi\'nski Gasket in \cite{VFS}.
Liang \cite{Liang} has established the relationship between the bounded variation function on a compact unit interval and its Riemann–Liouville fractional integral. On closed intervals, Liang \cite{ Liang1} has investigated the  fractal dimensions of fractional calculus of continuous functions. Furthermore, the fractal dimension of the fractional integral of continuous functions having a unit fractal dimension on the compact unit interval has been studied by Liang \cite{Liang11}. Recently, Chandra and Abbas have studied the fractal dimension of the graph of the mixed Riemann-Liouville fractional integral for several continuous function on rectangular region in \cite{ssna}. The author of \cite{Feng} has studied the fractal interpolation surface on the rectangular domain and discussed the oscillation and variation aspects of bivariate continuous functions. In \cite{pandey2022some,Megha2022}, Pandey et.al., introduced the notion of multivariate $\alpha$-fractal function.  We intend to present a technique for constructing a multivariate fractal function on hypercube, as proposed by Verma \cite{Jha}. In particular, we focus on fractal dimensions of multivariate fractal functions.

The paper is organised as follows. In the next section, we noted the preliminaries and notations which will be required throughout the study.
Section \ref{section2} is dedicated to the construction of multivariate fractal function for a given data point (subsection \ref{subsec1}) and the establishing the existence of $\alpha$-fractal function corresponding to the given continuous function defined on $I^q$ (subsection \ref{FIF}). Section \ref{Bdm} is reserved for the study of box dimension and Hausdorff dimension of the constructed multivariate $\alpha$-fractal function and based on the clue we noticed that if the given function is H\"older continuous, then for suitable parameters the associated $\alpha$-fractal function is also H\"older continuous. Moreover, in Subsection \ref{Restrictions}, we studied the restrictions of $\alpha$-fractal function on the axis. In Section \ref{Mrl}, we proved that mixed Riemann-Liouville fractional integral of fractal interpolation function satisfies the self-referential equation.

\section{Notation and Prelude}
The following are the notations we shall be using throughout the paper:
\begin{itemize}
    \item $\mathbb{N}$ : set of natural numbers.
    \item $\mathbb{R}$ : set of real numbers.
    \item $I$ : unit interval $[0,1]$.
    \item $I^q=I\times \cdots \times I$: $q$-times Cartesian product of $[0,1]$.
    \item $\mathcal{C}(I^q)$ : set of all continuous functions on $I^q$.
\end{itemize}

\begin{definition}
  Let $X$ and $Y$ be metric spaces and $h:X\rightarrow Y$ be a function from $X$ to $Y$, then graph of $h$ is defined as,
 \begin{equation}\label{Gf1}
   \mathcal{G}(h)=\{(x,h(x))~:~ x\in X \}.
 \end{equation}
 
 \begin{definition}
      Consider $U \ne \emptyset$ and $U \subset \mathbb{R}^q $ under the Euclidean norm $\|.\|_2$. Furthermore, the diameter of $U$ denoted by $|U|$ is defined as follows: 
      \[|U|=\sup\big\{\|x-y\|_2: x,y \in U\big\}.\]
      If $\{U_i\}$ is a countable  collection of sets of diameter at most $\delta$ that cover $F,$ we say that $\{U_i\}$ is a $\delta$-cover of $F.$ Assume that $F \subset \mathbb{R}^q$ and $s$ is a non-negative real number. For any $\delta >0,$ we define 
      \[H^s_{\delta}(F)= \inf \left\{\sum_{i=1}^{\infty}|U_i|^s:  \{U_i\} \text{ is a } \delta \text{-cover of } F \right\}.\] 
      The $s$-dimensional Hausdorff measure of $F$ is defined as $H^s(F)= \lim_{\delta \rightarrow 0}H^s_{\delta}(F).$
    \end{definition}
     \begin{definition}
     Let $F \subseteq \mathbb{R}^n$ and $s \ge 0.$ The Hausdorff dimension of $F$ is 
     \[\dim_H(F)=\inf\{s:H^s(F)=0\}=\sup\{s:H^s(F)=\infty\}.\]
     \end{definition}
    
     \begin{definition} Let $F \ne \emptyset$ be a bounded subset of $\mathbb{R}^n$ and let $N_{\delta}(F)$ be the smallest number of sets of diameter at most $\delta$ which can cover $F.$ The lower box dimension and upper box dimension of $F$, respectively, are defined as \[\underline{\dim}_B(F)=\varliminf_{\delta \rightarrow 0} \frac{\log N_{\delta}(F)}{- \log \delta},\]
     and
     \[\overline{\dim}_B(F)=\varlimsup_{\delta \rightarrow 0} \frac{\log N_{\delta}(F)}{- \log \delta}.\]
     Whenever the aforementioned values are equal, a common value is referred to as the box dimension of $F,$
      that is,  
      \[\dim_B(F)=\lim_{\delta \rightarrow 0} \frac{\log N_{\delta}(F)}{- \log \delta}.\]
     \end{definition}
To study about these dimensions in more details it is recommend the reader to follow \cite{Fal}.

\end{definition}

\section{Multivariate fractal interpolation function and $\alpha$-fractal function}\label{section2}
\subsection{Multivariate Fractal Interpolation Function}\label{subsec1}
The data of interpolation is given as 
\[\left\{\left(x_{1,i_1},\ldots, x_{q,i_q}, z_{(1,i_1),\ldots,(q,i_q)}\right): i_1=0,1,\ldots, M_1; ~ \ldots ~; i_q=0,1,\ldots, M_q\right\},\]
such that $0=x_{k,0}<\cdots < x_{k,M_k}=1$ for each $I_k$, where $I_k$ denotes the $k^{th}$ unit interval in $I^q$.\\
For our convenience, we use the following notations in the rest of the study:\\
$ \Sigma_k=\{1,2,\ldots,k\},$ $\Sigma_{k,0}=\{0,1,\ldots, k \},$ $\partial  \Sigma_{k,0}=\{0,k\} $ and int$\Sigma_{k,0}=\{1,2,\ldots,k-1\}.$ Further, a net $\Delta$ on $I^q$ is defined as follows:
\[\Delta:=\left\{(x_{1,i_1}
, \ldots , x_{q,i_q}) \in I ^q: i_k \in \Sigma_{M_k,0}, ~ 0= x_{k,0} <\ldots < x_{k,M_k} =1,~ k \in \Sigma_{q}\right\}.\]
For each $i_k \in  \Sigma_{M_{k}}$, let us consider $I_{k,i_{k}} = [x_{k,i_{k}-1}, x_{k,i_{k}}]$ and define contraction functions $u_{k,i_{k}}
: I_k \rightarrow I_{k,i_k}$ such that
\begin{equation}\label{affine}
    \begin{aligned}
        u_{k,i_k}(x_{k,0})=&x_{k,i_k-1}, \quad \quad u_{k,i_k}(x_{k,M_k})=x_{k,i_k}, \text{  if $i_k$ is odd,}\\
 u_{k,i_k}(x_{k,0})=&x_{k,i_k}, \quad \quad u_{k,i_k}(x_{k,M_k})=x_{k,i_k-1}, \text{ if $i_k$ is even, and }\\
 \left\lvert u_{k,i_k}(x)-u_{k,i_k}(y) \right\rvert &\leq \mu_{k,i_k}\lvert x-y\rvert \text{ for all } x,y \in I_k, \text{ and } 0< \mu_{k,i_k}<1.
    \end{aligned}
\end{equation}
It is easy to observe from the definition of $u_{k,i_k}$ that 
\[u_{k,i_{k}}^{-1}(x_{k,i_{k}})=u_{k,i_{k}+1}^{-1}(x_{k,i_k}) \text{ for all } i_k\in \text{int}\Sigma_{k,0} \text{ and } k\in \Sigma_q.\]
Define a function $\eta : \mathbb{Z}\times \{0, M_1, \ldots, M_q\}\rightarrow \mathbb{Z}$ by 
\begin{equation*}
    \begin{cases}
    \eta(i,0)=i-1, \quad \eta(i,M_1)=\cdots =\eta(i,M_q)=i, &\text{ when $i$ is odd,}\\
    \eta(i,0)=i, \quad \eta(i,M_1)=\cdots =\eta(i,M_q)=i-1, & \text{ when $i$ is even.}
    \end{cases}
\end{equation*}
 Set $K=I^q\times \mathbb{R}$. For each $(i_1,\ldots, i_q)\in \overset{q}{\underset{k=1}{\prod}} \Sigma_{M_k}$, let $F_{i_1,\ldots, i_q}:K\rightarrow \mathbb{R}$ be a continuous function satisfying
 \[F_{i_1\cdots i_q}\left(x_{1,k_1},\ldots, x_{q,k_q}, z_{(1,k_1)\cdots(q,k_q)}\right)=z_{(1,\eta(i_1,k_1))\cdots (q,\eta(i_q,k_q))} \text{ for all } (k_1,\ldots, k_q)\in \prod_{j=1}^{q}\partial \Sigma_{M_j,0}, \text{ and }\]
 \[\left\lvert F_{i_1\cdots i_q}(x_1,\ldots,x_q, z^{*})-F(x_1,\ldots, x_q, z^{**})\right \rvert \leq \gamma_{i_1\cdots i_q}\lvert z^{*}-z^{**}\rvert \text{ for all }(x_1,\ldots, x_q)\in I^{q},~ z^{*}, z^{**}\in \mathbb{R},\]
 and  $0<\gamma_{i_1\cdots i_q}<1~$ is any given constant.\\
For each $(i_1,\ldots, i_q)\in \overset{q}{\underset{k=1}{\prod}} \Sigma_{M_k}$, define a function $\Omega_{i_1\cdots i_q}:K\rightarrow \overset{q}{\underset{k=1}{ \prod}}I_{k,i_k}\times \mathbb{R}$ such that
\begin{equation}\label{IFS}
  \Omega_{i_1\cdots i_q}(x_1,\ldots, x_q, z)=\left(u_{1,i_1}(x_1),\ldots, u_{q,i_q}(x_q),F_{i_1\cdots i_q}(x_1,\ldots, x_q,z)\right).  
\end{equation}
Then, with the assumptions on $F_{i_1\cdots i_q}$ and $u_{k,i_k}$, for each $k\in  \Sigma_q$, one can prove that $\Omega_{i_1\cdots i_q}$ is a contraction function, hence $\left\{K,~\Omega_{i_1\cdots i_q}: (i_1,\ldots,i_q)\in \overset{q}{\underset{k=1}{ \prod}} \Sigma_{M_k}\right\}~$ is an IFS.
\begin{theorem}\label{matching}
Consider $\left\{K,\Omega_{i_1\cdots i_q}~: (i_1,\ldots,i_q)\in \overset{q}{\underset{k=1}{ \prod}} \Sigma_{M_k}\right\}$ be the IFS as defined in \eqref{IFS}. For every $k=1,\ldots,q$, suppose $F_{i_1\cdots i_q}$ satisfies the following matching conditions: for all $i_k\in \text{int} \Sigma_{M_{k},0},\text{ and } j\neq k,~i_j\in  \Sigma_{M_j} \text{ and } x_{k}^{*}=u_{k,i_k}^{-1}(x_{k,i_k})=u_{k,i_k+1}^{-1}(x_{k,i_k}),$
\begin{equation}
    F_{i_1\cdots i_q}(x_1,\ldots, x_{k}^{*}, x_{k+1},\ldots, x_q, z)=F_{i_1\cdots i_{k-1},i_{k}+1 i_{k+1}\cdots i_q}(x_1,\ldots, x_{k}^{*}, x_{k+1},\ldots, x_q, z)
\end{equation}
for all $x_j\in I_j, ~ z\in \mathbb{R}$. Then, there is a unique continuous function $\mathcal{A}: I^{q}\rightarrow \mathbb{R}$ such that \[\mathcal{A}(x_{1,i_1},\ldots, x_{q,i_q})=z_{(1,i_1)\cdots (q,i_q)} \text{ for all } (i_1,\ldots, i_q)\in \overset{q}{\underset{k=1}{ \prod}} \Sigma_{M_k,0},\]
and $H= \bigcup\left\{\Omega_{i_1\cdots i_q}(H): (i_1,\ldots,i_q)\in \overset{q}{\underset{k=1}{ \prod}} \Sigma_{M_k}\right\}$, where $H=\mathcal{G}(\mathcal{A})$.
We pronounce $H$ the FIS and $\mathcal{A}$ the FIF with respect to the IFS defined in \eqref{IFS}.
\end{theorem}
The proof of the above theorem can be accomplished by using the technique similar to Ruan and Xu \cite{Ruan}, hence it is omitted here.

\subsubsection{Multivariate Fractal Interpolation surfaces}\label{subsubsec1}
For $(i_1,\ldots, i_q)\in \prod_{k=1}^{q}\Sigma_{M_k,0}$, define $F_{i_1\cdots i_q}: K\rightarrow \mathbb{R}$ such that 
\[F_{i_1\cdots i_q}(x_1,\ldots, x_q, z)=\delta z+\mathcal{B}_{i_1\cdots i_q}(x_1,\ldots,x_q),\]
where
 \[\mathcal{B}_{i_1\cdots i_q}(x_{1,k_1},\ldots x_{q,k_q})=z_{(1,\eta(i_1,k_1)),\cdots, (q,\eta(i_q,k_q))}-\delta z_{(1,k_1)\cdots (q,k_q)} \text{ for all } (k_1,\ldots, k_q)\in \prod_{j=1}^{q}\partial \Sigma_{M_j,0}\] and $\lvert \delta \rvert<1.$
Then, observe that $F_{i_1\cdots i_q}$ satisfies the conditions of Theorem \ref{matching}.Now for all $(i_1,\ldots, i_q)\in \overset{q}{\underset{k=1}{ \prod}} \Sigma_{M_k}$, define Read-Bajraktarevi\'c (RB) operator, $T:\mathcal{C}(I^q)\rightarrow \mathcal{C}(I^q)$ such that
\[Tg(x_1,\ldots, x_q)=F_{i_1\cdots i_q}\left(u_{i,i_1}^{-1}(x_1),\ldots, u_{q,i_q}^{-1}(x_q), g\left(u_{1,i_1}^{-1}(x_1),\ldots,u_{q,i_q}^{-1}(x_q)\right)\right).\]
Then, the multivariate FIF,  $\mathcal{A}$ will be the unique fixed point of $T$. Then, for all $(x_1,\ldots, x_q)\in \overset{q}{\underset{k=1}{ \prod}}I_{k,i_k}$ and $(i_1,\ldots, i_q)\in \overset{q}{\underset{k=1}{ \prod}}~ \Sigma_{M_k}$,  $\mathcal{A}$ will satisfy the following self-referential equation:
\[
    \mathcal{A}(x_1,\ldots, x_q)=F_{i_1\cdots i_q}\left(u_{i,i_1}^{-1}(x_1),\ldots, u_{q,i_q}^{-1}(x_q), \mathcal{A}\left(u_{1,i_1}^{-1}(x_1),\ldots,u_{q,i_q}^{-1}(x_q)\right)\right.
\]
Therefore, we have
\begin{equation}\label{interpolation}
    \mathcal{A}(x_1,\ldots, x_q)=\delta \mathcal{A}\left(u_{1,i_1}^{-1}(x_1),\ldots,u_{q,i_q}^{-1}(x_q)\right)+\mathcal{B}_{i_1\cdots i_q}\left(u_{1,i_1}^{-1}(x_1),\ldots,u_{q,i_q}^{-1}(x_q)\right).
\end{equation}

\begin{note}\label{Naffine}
For every $x\in I_{k,{i_{k}}}$, define
$u_{k,i_{k}}(x)=a_{k,i_{k}}(x)+b_{k,i_{k}}$ such that
\begin{align*}
    a_{k,i_{k}}&=\frac{x_{k,i_{k}}-x_{k,i_{k-1}}}{x_{k,M_{k}}-x_{k,0}} \quad \text{and} \quad b_{k,i_{k}}= \frac{x_{k,i_{k}}x_{k,M_{k}}-x_{k,i_{k-1}}x_{k,0}}{x_{k,M_{k}}-x_{k,0}}, \text{  if $i_k$ is odd}\\
    a_{k,i_{k}}&=\frac{x_{k,i_{k-1}}-x_{k,i_{k}}}{x_{k,M_{k}}-x_{k,0}} \quad \text{and} \quad b_{k,i_{k}}=\frac{x_{k,i_{k-1}}x_{k,M_{k}}-x_{k,i_{k}}x_{k,0}}{x_{k,M_{k}}-x_{k,0}}, \text{  if $i_k$ is even}.
\end{align*}
 Then, obviously $u_{k,i_{k}}$ satisfies equation \eqref{affine}.
\end{note}

%%%%%%%%%%%%%%%%%%%%%%%%%%%%%%%%%%%%%%%%%%%%%%%%%%%%%%%%%%%%%%%%%%%%%%%%%%%%%%%%%%%%%%%%%%%%%%%%%%%%%%%%%%%%%%%%%%%%%%%%%%%%%%%%%%%%%%%%%%%%%%%%%%%%%%%%%%%%%%%%%%%%%%%%%%%%%%%%%%%%%%%%%%%%%

\subsection{$\alpha$- Fractal Functions} \label{FIF}
Let $f \in \mathcal{C}(I^q)$, then using the notions defined in Subsection \ref{subsec1}, let us denote
\[Q_{i_1\ldots i_q}(x_1,\ldots,x_q):= \left(u_{1,i_1}^{-1}(x_1),\ldots,u_{q,i_q}^{-1}(x_q)\right), \text{ for }
(x_1,\ldots,x_q) \in \overset{q}{\underset{k=1}{ \prod}}I_{k,i_k}.\]
Let $\alpha \in  \mathcal{C}(I^q)$ be such that $\|\alpha\|_{\infty}<1.$ Assume further that $s \in  \mathcal{C}(I^q)$ satisfies \[s(x_{1,j_1},\ldots,x_{q,j_q})=f(x_{1,j_1},\ldots,x_{q,j_q}) \text{ for all } (j_1, \ldots,j_q) \in
\overset{q}{\underset{k=1}{ \prod}} \partial  \Sigma_{M_k,0}.\] 
For each $(i_1, \dots, i_q) \in \overset{q}{\underset{k=1}{ \prod}}  \Sigma_{M_{k}}$, define $\mathcal{H}_{i_{1}\dots i_{q}}: K \rightarrow \mathbb{R}$ as follows
\begin{align}
 &\mathcal{H}_{i_1 \dots i_q}(y_1, \dots, y_q, z) \nonumber \\
 =&f\left(u_{1,i_1}(y_1), \dots, u_{q,i_{q}}(y_q)\right)+\alpha\left(u_{1,i_1}(y_1), \dots, u_{q,i_{q}}(y_q)\right) \left(z-s(y_1, \dots, y_q)\right),\label{falpha}
 \end{align}
 Then, $W_{i_1 \dots i_q}: K \rightarrow \overset{q}{\underset{k=1}{ \prod}}I_{k,i_k}\times \mathbb{R}$ defined by 
\[
W_{i_1 \dots i_q}(x_1, \dots, x_q, z)=\left(u_{1,i_1}(x_1), \dots, u_{q,i_{q}}(x_q), \mathcal{H}_{i_1 \dots i_q}(x_1, \dots, x_q, z)\right),
\]
is an IFS on $K$ and notice that $\mathcal{H}_{i_1 \dots i_q}$ satisfies the conditions of Theorem \ref{matching}, therefore we have the following:
\begin{theorem}
 Let $\left\{K, W_{i_1 \dots i_q}:(i_1, \dots, i_q) \in \overset{q}{\underset{k=1}{ \prod}}  \Sigma_{M_{k}}\right\}$ be the IFS. Then, there exists a unique continuous function $f^{\alpha}_{\Delta,s}: I^q \rightarrow$ $\mathbb{R}$ such that $f^{\alpha}_{\Delta,s}\left(x_{1,i_1},\dots,x_{q.i_q}\right)=f\left(x_{1,i_1},\dots,x_{q.i_q}\right)$ for all $(i_1, \dots, i_q) \in \overset{q}{\underset{k=1}{ \prod}} \Sigma_{ M_{k},0}$ and $G= \bigcup\left\{W_{i_1\cdots i_q}(G): (i_1,\ldots,i_q)\in \overset{q}{\underset{k=1}{ \prod}} \Sigma_{M_k,0}\right\}, $ where $\mathcal{G}(f^{\alpha}_{\Delta,s})=G$.
\end{theorem}
Now define a RB Operator, $T:\mathcal{C}(I^q)\rightarrow \mathcal{C}(I^q)$ such that 
\[T(h)(x_1,\ldots, x_q)=\mathcal{H}_{i_1 \dots i_q}\left(u_{1,i_1}^{-1}(x_1),\dots,u_{q,i_q}^{-1}(x_q), h\left(u_{1,i_1}^{-1}(x_1),\ldots,u_{q,i_q}^{-1}(x_q)\right)\right),\]
then there exists a fixed point of this operator, $f^{\alpha}_{\Delta,s}$ known as $\alpha$-fractal function and it satisfies the following self-referential equation,
\[f^{\alpha}_{\Delta,s}(x_1, \dots, x_q)=\mathcal{H}_{i_1 \dots i_q}\left(u_{1,i_1}^{-1}(x_1),\dots,u_{q,i_q}^{-1}(x_q), f^{\alpha}_{\Delta,s}\left(u_{1,i_1}^{-1}(x_1),\ldots,u_{q,i_q}^{-1}(x_q)\right)\right)\]
for all $(x_1, \dots, x_q) \in \overset{q}{\underset{k=1}{ \prod}}I_{k,i_k}.$
That is,
 \begin{equation}\label{alphafrac}
 \begin{aligned}
   &f^{\alpha}_{\Delta,s}(x_1,\ldots,x_q)\\
   =& f(x_1,\ldots,x_q)+\alpha\left(u_{1,i_1}^{-1}(x_1),\ldots,u_{q,i_q}^{-1}(x_q)\right)\left[ f^{\alpha}_{\Delta,s}\left(u_{1,i_1}^{-1}(x_1),\ldots,u_{q,i_q}^{-1}(x_q)\right)\right.\\
   &\hspace{5cm}\left.-s\left(u_{1,i_1}^{-1}(x_1),\ldots,u_{q,i_q}^{-1}(x_q)\right)\right]
   \end{aligned}
\end{equation}
for $(x_1,\dots,x_q) \in \overset{q}{\underset{k=1}{ \prod}} I_{k,i_k}.$

\noindent
$\bullet$ Denote $f^{\alpha}_{\Delta,s}$ as $f^{\alpha}$ if there is no ambiguity.\\
$\bullet$ For our convenience, write the metric \[d_{I^q}((x_1,\dots,x_q),(y_1,\dots,y_q)):=\sqrt{(x_1-y_1)^2+\cdots+(x_q-y_q)^2},\]
where$(x_1,\dots,x_q), ~ (y_1,\dots,y_q) \in I^q.$

\begin{definition}
    A function $f: I^q \rightarrow \mathbb{R}$ is said to be H\"{o}lder continuous with exponent $\sigma$ and H\"{o}lder constant $K_f$ if 
    \[ |f(x_1,\dots,x_q)-f(y_1,\dots,y_q)| \le K_f d_{I^q}((x_1,\dots,x_q),(y_1,\dots,y_q))^\sigma\]
    for all $(x_1,\dots,x_q),(y_1,\dots,y_q)\in I^q, \text{and for some}~ K_f > 0.$\\
    Define the H\"{o}lder space 
\[ \mathcal{H}^{\sigma}(I^q ) := \{g:I^q \rightarrow \mathbb{R}: ~\text{g is H\"{o}lder continuous with exponent}~\sigma \} .\]
       If we equip the space $\mathcal{H}^{\sigma}(I^q)$ with norm $ \|g\|:= \|g\|_{\infty} +[g]_{\sigma}$ where
       \[[g]_{\sigma} = \sup_{(x_1,\dots,x_q)\ne (y_1,\dots,y_q)} \frac{|g(x_1,\dots,x_q)-g(y_1,\dots,y_q)|}{d_{I^q}((x_1,\dots,x_q),(y_1,\dots,y_q))^{\sigma}},\]
       then it forms a Banach space.
    \end{definition}

\section{Box dimension of the graph of multivariate fractal function} \label{Bdm}
By estimating the oscillation of functions, we shall seek to compute the box dimension of the graph of the multivariate fractal function. The main ideas in \cite{ MF4, Ruan2} are modified and applied in our method.
\begin{lemma} \label{SDBl}
      Let $f : I^q \rightarrow \mathbb{R}$ be continuous on $I^q.$ Then, $\dim_H \big(\mathcal{G}(f)\big) \ge q.$
\end{lemma}
      \begin{proof}
      Define a Lipschitz function $T : \mathcal{G}(f) \rightarrow I^q $ by $T\big((x_1,\ldots, x_q,f(x_1,\ldots, x_q))\big)=(x_1,\ldots, x_q).$  Recall that if $f:A\subseteq \mathbb{R}^n \to \mathbb{R}^m$ is a Lipschitz function,  then $\dim_{H}\big( f(A) \big) \le \dim_{H}(A).$ Hence, we get $\dim_H \big(T(\mathcal{G}(f))\big) \le \dim_H\big(\mathcal{G}(f)\big).$ It is straightforward to verify that the function $T$ is surjective and thus the conclusion.
      \end{proof}
    \begin{definition}
    Let $A \subset \mathbb{R}^q$ be a closed bounded hypercube and $f: A \to \mathbb{R}$. The oscillation of  $f$ over the rectangle $A$ is defined as
    \[ Osc_f[A]= \sup_{(x_1,\ldots, x_q),(y_1,\ldots, y_q) \in A} |f(x_1,\ldots, x_q)-f(y_1,\ldots, y_q)|.\]
    \end{definition}
    Next, we shall provide a multivariate analogue of Proposition $11.1$  in Falconer \cite{Fal}.
    \begin{lemma} \label{BBVL2}
    Let $f :I^q \rightarrow \mathbb{R}$ be a continuous function. For $M_1< \ldots < M_q < M $, let us take  cuboid of size $  \frac{1}{M_1^m} \times \dots \times   \frac{1}{M_q^m} \times  \frac{1}{M^m} $ for some $M_k,M \ge 2$, where $1\le k \le q$ and for some $m \in \mathbb{N}.$ Let $ \delta= \frac{1}{M^m}$ be the height of the cuboid. If $N_{\delta}(\mathcal{G}(f))$ is the number of $\delta$-cuboid that intersect the graph of $f,$  then
    \begin{align*}
       \sum_{i_q=1}^{M_q^m}\dots \sum_{i_1=1}^{M_1^m} \max\Big\{ 1, \left \lceil { M^m Osc_f[A_{i_{1}\dots i_{q}}]} \right \rceil \Big\} \le& N_{\delta}(\mathcal{G}(f))\\
       \le& [M_1 \dots M_q]^m  + \sum_{i_1=1}^{M_1^m} \dots \sum_{i_q=1}^{M_q^m}  \left \lceil { M^m Osc_f[A_{i_1\dots\i_q}]} \right \rceil, 
    \end{align*}
 where $\left \lceil {.}\right \rceil $ denotes the ceiling function.
    \end{lemma}
    \begin{proof}
Using the continuity of $f$, it follows that the number of $\delta$-cuboids in the part above $A_{i_1 \dots i_q}$ that intersect the graph of $f$ is at least $\max\Big\{ 1, \left \lceil {M^m Osc_f[A_{i_1 \dots i_q}]} \right \rceil \Big\} $ and at most  $1+ \left \lceil {M^m Osc_f[A_{i_1 \dots i_q}]}\right \rceil$. The appropriate bounds are obtained by summing all of these components.
\end{proof}
Now we provide a few concerned most fundamental concepts and notations. We shall assume for the purpose of simplicity $I= I_k=[0,1]$, where $k \in \Sigma_{q}$.\\
 Let $\left\{(x_{1, i_1}, \dots, x_{q,i_q}) \in \mathbb{R}^q : i_k\in \Sigma_{M_k, 0}, k \in \Sigma_{q}\right\}$ such that $0=x_{k,0}< x_{k,1} < \dots <x_{k, M_k}=1$  with  $x_{k,i_k} - x_{k,i_k-1}=  \frac{1}{M_k}, \text{ where } k \in  \Sigma_{q} $. Further, we denote
 \begin{align*}
 & D_{i_1 \dots i_q}= u_{1,i_1}(I_1) \times \dots \times u_{q,i_q} (I_q), \quad D_{(i_1)_1 \dots (i_1)_{m}  \dots (i_q)_1 \dots(i_q)_{m} }= u_{(i_1)_1 \dots (i_1)_{m} }(I_1) \times \dots \times u_{(i_q)_1 \dots(i_q)_{m}}(I_q),\\ 
 & u_{(i_k)_1 \dots (i_k)_{m} }=u_{k,(i_k)_1} \circ \dots \circ u_{k,(i_k)_m},~ \text{for}~ 1\le k \le q,
 \end{align*}
 and define $ G_{i_1 \dots i_q}:=\{(x_1, \dots, x_q,f^{\alpha}(x_1, \dots, x_q)): (x_1, \dots, x_q ) \in D_{i_1 \dots i_q} \}$ and
 \begin{equation*}
 \begin{aligned}
 \overline{\alpha}_{i_1 \dots i_q} =&~ \max\{|\alpha(x_1, \dots, x_q)|:(x_1, \dots, x_q ) \in D_{i_1 \dots i_q} \},\\ \alpha_{\min}= & \min\{|\alpha(x_1, \dots, x_q)|:(x_1, \dots, x_q ) \in I^q\}, \\\alpha_{\max}= & \max\{|\alpha(x_1, \dots, x_q)|:(x_1, \dots, x_q ) \in I^q\}.
    \end{aligned}
 \end{equation*}
    \begin{remark}
     Let $ \overline{\gamma}= \overset{M_1,\dots,M_q}{\underset{i_{1},\dots,i_{q}=1}{\sum}} \overline{\alpha}_{i_{1},\dots, i_{q}}.$ For $m \in \mathbb{N},$ we have 
     \[\overline{\gamma}^m= \sum_{(i_{k,1}\dots, i_{k,q}) \in  \Sigma_{M_{1}} \times \dots \times  \Sigma_{M_{q}}} \overline{\alpha}_{i_{m,1},\dots, i_{m,q}} \dots \overline{\alpha}_{i_{1,1},\dots, i_{1,q}}.\]
    \end{remark}
    \begin{definition}
    A function $f: I^{q} \rightarrow \mathbb{R}$ is said to be H\"{o}lder continuous with exponent $\sigma$ and H\"{o}lder constant $K_f$ if 
    \[ |f(x_{1}, \dots, x_{q})-f(y_1,\dots,y_q)| \le K_f \|(x_{1}, \dots, x_{q})-(y_1,\dots,y_q)\|_2^\sigma \text{ for all } (x_{1}, \dots, x_{q}),(y_1,\dots,y_q) \in I^q, \] 
    and for some $K_f > 0.$
    \end{definition}
    \begin{lemma}\label{lemma3}
 $\alpha$-fractal function $f^\alpha$ is constructed using $f, \alpha$ and $s$ that are H\"older continuous, corresponding to these consider the $ \sigma_1, \sigma_2, \sigma_3$ as an exponents and $K_f, K_\alpha, K_s$ as a H\"older constant  respectively.
     Furthermore, we suppose that $M_{q}<\dots <M_{1}.$  Then, we have
     \begin{equation}
    \begin{aligned}
       Osc_{F_{(i_{1})_m\dots  (i_{q})_m}}[G_{(i_1)_1 \dots (i_1)_{m-1}  \dots (i_q)_1 \dots(i_q)_{m-1} }] \le& \overline{\alpha}_{(i_{1})_m\dots  (i_{q})_m} Osc_{f^{\alpha}}[D_{(i_1)_1 \dots (i_1)_{m-1}  \dots (i_q)_1 \dots(i_q)_{m-1} }] \\
     &+\left( \frac{K^*}{M_{1}^{m \sigma}} + \frac{\overline{\alpha}_{(i_{1})_m\dots  (i_{q})_m} K_s}{M_{1}^{(m-1) \sigma}} \right)q^{\sigma/2},          
     \end{aligned}
     \end{equation}
      where  $\sigma=\min\{\sigma_1, \sigma_2,\sigma_3\}$ and $K^*:= K_\alpha\big(\|f\|_\infty+ \|s\|_\infty\big)+K_f$.
    \end{lemma}

    \begin{proof}
    Recall that 
    \begin{equation}
    \begin{aligned}
F_{i_{1} \dots i_{q}}(x_{1}, \dots, x_{q},z) =&\alpha \big( u_{1,i_{1}}(x_1),\dots, u_{q, i_{q}}(x_q)\big)z+f\big(u_{1, i_{1}}(x_1),\dots, u_{q, i_{q}}(x_q)\big)\\ 
   &- \alpha \big(u_{1, i_{1}}(x_1),\dots, u_{q, i_{q}}(x_q)\big)s(x_1,\dots, x_q).
    \end{aligned}
    \end{equation}
    Let $(x_{1}, \dots, x_{q}) , (y_{1}, \dots, y_{q}) \in I^{q} .$ We have
    \begin{align*}
          &\big|F_{(i_{1})_1 \dots  (i_{q})_1}\big(x_{1}, \dots, x_{q}, f^{\alpha}(x_{1}, \dots, x_{q})\big)- F_{(i_{1})_1\dots  (i_{q})_1}\big(y_{1}, \dots, y_{q}, f^{\alpha}(y_{1}, \dots, y_{q})\big) \big| \\
          =~ & \big| \alpha \big( u_{1,(i_{1})_1}(x_1),\dots, u_{q,(i_{q})_1}(x_q)\big)f^{\alpha}(x_{1}, \dots, x_{q})+f\big( u_{1,(i_{1})_1}(x_1),\dots, u_{q,(i_{q})_1}(x_q)\big)\\
          &- \alpha\big( u_{1,(i_{1})_1}(x_1),\dots, u_{q,(i_{q})_1}(x_q)\big)s(x_1,\dots,x_q)  - \alpha\big( u_{1,(i_{1})_1}(y_1),\dots, u_{q,(i_{q})_1}(y_q)\big)f^{\alpha}(y_{1}, \dots, y_{q})\\
          &+f\big( u_{1,(i_{1})_1}(y_1),\dots, u_{q,(i_{q})_1}(y_q)\big)- \alpha\big( u_{1,(i_{1})_1}(y_1),\dots, u_{q,(i_{q})_1}(y_q)\big)s(y_{1}, \dots, y_{q})\big|\\
          \le ~& \big| \alpha\big( u_{1,(i_{1})_1}(x_1),\dots, u_{q,(i_{q})_1}(x_q)\big) \big| |f^{\alpha}(x_{1}, \dots, x_{q}) - f^{\alpha}(y_{1}, \dots, y_{q})|\\
       \end{align*}
     \begin{align*}
&+\big| \alpha\big( u_{1,(i_{1})_1}(x_1),\dots, u_{q,(i_{q})_1}(x_q)\big) \big| |s(x_{1}, \dots, x_{q}) - s(y_{1}, \dots, y_{q})| \\&+
          \big(|f^{\alpha}(y_{1}, \dots, y_{q})|+|s(y_{1}, \dots, y_{q})| \big) \big|\alpha\big( u_{1,(i_{1})_1}(x_1),\dots, u_{q,(i_{q})_1}(x_q)\big) 
          - \alpha\big( u_{1,(i_{1})_1}(y_1),\dots, u_{q,(i_{q})_1}(y_q)\big)  \big|\\&+
          \big|f\big( u_{1,(i_{1})_1}(x_1),\dots, u_{q,(i_{q})_1}(x_q)\big)- f\big( u_{1,(i_{1})_1}(y_1),\dots, u_{q,(i_{q})_1}(y_q)\big) \big|\\
          \le~ & \overline{\alpha}_{(i_{1})_1\dots  (i_{q})_1} |f^{\alpha}(x_{1}, \dots, x_{q}) - f^{\alpha}(y_{1}, \dots, y_{q})|\\ &+ \overline{\alpha}_{(i_{1})_1\dots  (i_{q})_1} K_s \|(x_{1}, \dots, x_{q})-(y_{1}, \dots, y_{q})\|_2^{\sigma_3}\\&+
         K_{\alpha} \big(\|f^{\alpha}\|_{\infty}+\|s\|_{\infty} \big) \big\|\big( u_{1,(i_{1})_1}(x_1),\dots, u_{q,(i_{q})_1}(x_q)\big) - \big( u_{1,(i_{1})_1}(y_1),\dots, u_{q,(i_{q})_1}(y_q)\big)  \big\|_2^{\sigma_2}\\&+
         K_f \big\|\big( u_{1,(i_{1})_1}(x_1),\dots, u_{q,(i_{q})_1}(x_q)\big) - \big( u_{1,(i_{1})_1}(y_1),\dots, u_{q,(i_{q})_1}(y_q)\big)  \big\|_2^{\sigma_1}\\
        \le~ & \overline{\alpha}_{(i_{1})_1\dots  (i_{q})_1} Osc_{f^{\alpha}}[I^{q}]\\
         &+\overline{\alpha}_{(i_{1})_1\dots  (i_{q})_1} K_s\|(x_{1}, \dots, x_{q})-(y_{1}, \dots, y_{q})\|_2^{\sigma}\\&+ \Big(K_{\alpha} \big(\|f^{\alpha}\|_{\infty}+\|s\|_{\infty} \big)+ K_f\Big) \max\left\{\frac{1}{M_{q}^\sigma},\dots, \frac{1}{M_{1}^\sigma}\right\} \|(x_{1}, \dots, x_{q})-(y_{1}, \dots, y_{q})\|_2^{\sigma}\\
        \le~ & \overline{\alpha}_{(i_{1})_1\dots  (i_{q})_1} Osc_{f^{\alpha}}[I^{q}]+ \Big(\overline{\alpha}_{(i_{1})_1\dots  (i_{q})_1} K_s+\frac{K^*}{M_{1}^\sigma}\Big) q^{\sigma/2}.
\end{align*}

  Further iterations lead to the following result:
    \begin{equation}
        \begin{aligned}
          &Osc_{F_{(i_{1})_m\dots  (i_{q})_m}}[G_{(i_1)_1 \dots (i_1)_{m-1} (i_2)_1  \dots (i_2)_m \dots (i_q)_1 \dots(i_q)_{m-1 }}] \\
          \le & \overline{\alpha}_{(i_{1})_m\dots  (i_{q})_m} Osc_{f^{\alpha}}[D_{(i_1)_1 \dots (i_1)_{m-1} (i_2)_1  \dots (i_2)_m \dots (i_q)_1 \dots(i_q)_{m-1} }]+\left(\frac{K^*}{M_{1}^{m \sigma}}+ \frac{\overline{\alpha}_{(i_{1})_m\dots  (i_{q})_m} K_s}{M_{1}^{(m-1) \sigma}} \right)q^{\sigma/2},           
        \end{aligned}
    \end{equation}
completing the claim.
    \end{proof} The following theorem establishes bounds for the graph dimension of the multivariate $\alpha$-fractal function $f^{\alpha}.$
    \begin{theorem}\label{Th4.10}    
    Let $f^{\alpha}$ be the multivariate $\alpha$-fractal function constructed previously.
     \begin{enumerate}[(i)]
     \item If $ \alpha_{\min} >  \frac{1}{M_{1}^\sigma},$ then~
     $q \le \dim_H \big(\mathcal{G}(f^{\alpha})\big) \le \underline{\dim}_B \big(\mathcal{G}(f^{\alpha})\big) \le \overline{\dim}_B\big(\mathcal{G}(f^{\alpha})\big) \le 1+  \frac{\log (\overline{\gamma})}{\log (M)}.$
     \item If $  \frac{1}{M_{1}^\sigma} < \alpha_{\max} < 1,$ then
     $q \le \dim_H \big(\mathcal{G}(f^{\alpha})\big) \le \underline{\dim}_B \big(\mathcal{G}(f^{\alpha})\big) \le \overline{\dim}_B\big(\mathcal{G}(f^{\alpha})\big) \le q+1+  \frac{\log (\alpha_{\max})}{\log (M)}.$

     \item  If $ \alpha_{\max} \le  \frac{1}{M_{1}^\sigma},$ then
     $q \le \dim_H \big(\mathcal{G}(f^{\alpha})\big) \le \underline{\dim}_B \big(\mathcal{G}(f^{\alpha})\big) \le \overline{\dim}_B\big(\mathcal{G}(f^{\alpha})\big) \le q+1- \sigma.$
     \end{enumerate}
    \end{theorem}
    \begin{proof}
Note that the functional equation is satisfied by the function $f^{\alpha}$,
       \[f^{\alpha}(x_1, \dots, x_q)=F_{i_1 \dots i_q}\Big(u_{1,i_1}^{-1}(x_1),\ldots,u_{q,i_q}^{-1}(x_q),f^{\alpha}\big(u_{1,i_1}^{-1}(x_1),\ldots,u_{q,i_q}^{-1}(x_q)\big)\Big)\]
     for all $(x_1, \dots, x_q)~~\in \prod_{k=1}^{q} I_{k, i_k}$. For $(x_1, \dots, x_q),(y_1, \dots, y_q) \in D_{(i_1)_1 \dots (i_q)_1} ,$ we obtain
     \begin{align*}
       &\big|  f^{\alpha}(x_1, \dots, x_q) -f^{\alpha}(y_1, \dots, y_q)\big| \\
       =~ & \Bigg|F_{(i_1)_1 \dots (i_q)_1}\Big(u^{-1}_{1,(i_1)_1}(x_1),\dots,u^{-1}_{q,(i_q)_1}(x_q),f^{\alpha}\big(u^{-1}_{1,(i_1)_1}(x_1),\dots,u^{-1}_{q,(i_q)_1}(x_q)\big)\Big) \\
       &- F_{(i_1)_1 \dots (i_q)_1}\Big(u^{-1}_{1,(i_1)_1}(y_1),\dots,u^{-1}_{q,(i_q)_1}(y_q),f^{\alpha}\big(u^{-1}_{1,(i_1)_1}(y_1),\dots,u^{-1}_{q,(i_q)_1}(y_q)\big)\Big) \Bigg|\\
        \end{align*}
     \begin{align*}
        \le ~& \sup_{(x_1, \dots, x_q),(y_1, \dots, y_q) \in D_{(i_1)_1 \dots (i_q)_1}} \Big|F_{(i_1)_1 \dots (i_q)_1}\Big(u^{-1}_{1,(i_1)_1}(x_1),\dots,u^{-1}_{q,(i_q)_1}(x_q),\\
       & f^{\alpha}\big(u^{-1}_{1,(i_1)_1}(x_1),\dots,u^{-1}_{q,(i_q)_1}(x_q)\big)\Big) \\&- F_{(i_1)_1 \dots (i_q)_1}\Big(u^{-1}_{1,(i_1)_1}(y_1),\dots,u^{-1}_{q,(i_q)_1}(y_q),f^{\alpha}\big(u^{-1}_{1,(i_1)_1}(y_1),\dots,u^{-1}_{q,(i_q)_1}(y_q)\big)\Big) \Big|\\=~ & \sup_{(\tilde{x_1}, \dots, \tilde{x_q}),(\tilde{y_1}, \dots, \tilde{y_q}) \in I^q} \Big|F_{(i_1)_1 \dots (i_q)_1}\big(\tilde{x_1}, \dots, \tilde{x_q},f^{\alpha}(\tilde{x_1}, \dots, \tilde{x_q})\big)\\
       - &~~~~ F_{(i_1)_1 \dots (i_q)_1}\big(\tilde{y_1}, \dots, \tilde{y_q},f^{\alpha}(\tilde{y_1}, \dots, \tilde{y_q} )\big) \Big| \\ =~ &Osc_{F_{(i_1)_1 \dots (i_q)_1}}[\mathcal{G}(f^{\alpha})].
     \end{align*}
     Since $(x_1, \dots, x_q)$ and $(y_1, \dots, y_q)$ are arbitrary, we get
     \[Osc_{f^{\alpha}}[D_{(i_1)_1 \dots (i_q)_1}] \le Osc_{F_{(i_1)_1 \dots (i_q)_1}}[\mathcal{G}(f^{\alpha})].\]
     By succession, it follows that
     \[ Osc_{f^{\alpha}}[D_{{(i_1)_1 \dots (i_1)_m (i_2)_1  \dots (i_2)_m \dots (i_q)_1 \dots(i_q)_m }}] \le~ Osc_{F_{(i_{1})_m\dots  (i_{q})_m}}[G_{(i_1)_1\dots (i_1)_{m-1} \dots (i_q)_1 \dots (i_q)_{m-1}}].\]
Using Lemma \ref{lemma3}, we obtain
    \begin{align*}
     & Osc_{f^{\alpha}}[D_{(i_1)_1 \dots (i_1)_m \dots (i_q)_1 \dots(i_q)_m}]\\
     \le~ & \overline{\alpha}_{(i_1)_m \dots (i_q)_m}\dots \overline{\alpha}_{(i_1)_1 \dots (i_q)_1} Osc_{f^{\alpha}}[I^q]\\
     & +2^{\sigma/2} K^* \Big( \frac{1}{M_{1}^{m \sigma}} + \frac{\overline{\alpha}_{(i_1)_m \dots (i_q)_m} }{M_{1}^{(m-1) \sigma}} + \dots + \frac{\overline{\alpha}_{(i_1)_m \dots (i_q)_m} \dots \overline{\alpha}_{i_2 j_2}}{M_{1}^{\sigma}} \Big)\\& +
      2^{\sigma/2} K_s \Big( \frac{\overline{\alpha}_{(i_1)_m \dots (i_q)_m}}{M_{1}^{(m-1) \sigma}} + \frac{\overline{\alpha}_{(i_1)_m \dots (i_q)_m}\overline{\alpha}_{i_{m-1} j_{m-1}} }{M_{1}^{(m-2) \sigma}} + \dots + \overline{\alpha}_{(i_1)_m \dots (i_q)_m} \dots \overline{\alpha}_{(i_1)_1 \dots (i_q)_1}\Big).
     \end{align*}
    (a) If $ \alpha_{\min} >  \frac{1}{M_{1}^\sigma},$ then we get
     \begin{align*}
          & Osc_{f^{\alpha}}[D_{(i_1)_1 \dots (i_1)_m \dots (i_q)_1 \dots(i_q)_m }] \\
          \le ~& \overline{\alpha}_{(i_1)_m \dots (i_q)_m}\dots \overline{\alpha}_{(i_1)_1 \dots (i_q)_1} Osc_{f^{\alpha}}[I^q]\\& +
           \frac{2^{\sigma/2} K^* (\overline{\alpha}_{(i_1)_m \dots (i_q)_m} \dots \overline{\alpha}_{(i_1)_2 \dots (i_q)_2})} {M_{1}^\sigma} \Big( 1+ \frac{1}{ \alpha_{\min}M_{1}^{ \sigma}} +\dots + \frac{1 }{\alpha_{\min} ^{m-1}M_{1}^{(m-1) \sigma}}  \Big)\\& +
           2^{\sigma/2} K_s \overline{\alpha}_{(i_1)_m \dots (i_q)_m} \dots \overline{\alpha}_{(i_1)_1 \dots (i_q)_1}  \Big( 1+ \frac{1}{ \alpha_{\min}M_{1}^{ \sigma}} +\dots + \frac{1 }{\alpha_{\min} ^{m-1}M_{1}^{(m-1) \sigma}}  \Big) \\ 
           \le~ & \overline{\alpha}_{(i_1)_m \dots (i_q)_m}\dots \overline{\alpha}_{(i_1)_1 \dots (i_q)_1} \left(Osc_{f^{\alpha}}[I^q]+ \left(\frac{2^{\sigma/2} K^* }{M_{1}^\sigma \alpha_{\min}}+2^{\sigma/2} K_s \right)\left(\frac{1}{1 - \frac{1}{\alpha_{\min}M_{1}^{ \sigma}}} \right)\right).
          \end{align*}
          
      We write 
      \[Osc_{f^{\alpha}}[D_{(i_1)_1 \dots (i_1)_m (i_2)_1  \dots (i_2)_m \dots (i_q)_1 \dots(i_q)_m }] \le K_1 \overline{\alpha}_{(i_1)_m \dots (i_q)_m} \dots \overline{\alpha}_{(i_1)_1 \dots (i_q)_1},\] 
      where $K_1 = Osc_{f^{\alpha}}[I^q]+ \left( \frac{2^{\sigma/2} K^* }{M_{1}^\sigma \alpha_{\min}}+2^{\sigma/2} K_s \right)\left( \frac{1}{1 -  \frac{1}{\alpha_{\min}M_{1}^{\sigma}}}\right).$\\
     From Lemma \ref{BBVL2}, we have
     \begin{equation*}
          \begin{aligned}
           N_{\delta}(\mathcal{G}(f^{\alpha})) &\le [M\dots M]^m  + \sum_{ ((i_1)_k, \dots, (i_q)_k) \in \overset{q}{{\underset{k=1}{ \prod}}} \Sigma_{M_{k}}}  \left \lceil { M^m Osc_{f^{\alpha}}[D_{{(i_1)_1 \dots (i_1)_m (i_2)_1  \dots (i_2)_m \dots (i_q)_1 \dots(i_q)_m }}]} \right \rceil\\
           & \le M^{mq}  + \sum_{((i_1)_k, \dots, (i_q)_k) \in \overset{q}{{\underset{k=1}{ \prod}}} \Sigma_{M_{k}}} \left \lceil { M^m K_1 \overline{\alpha}_{(i_1)_m \dots (i_q)_m} \dots \overline{\alpha}_{(i_1)_1 \dots (i_q)_1}} \right \rceil \\
           & \le 2M^{mq} + M^m K_1 \sum_{ ((i_1)_k, \dots, (i_q)_k)   \in \overset{q}{{\underset{k=1}{ \prod}}} \Sigma_{M_{k}}}    \overline{\alpha}_{(i_1)_m \dots (i_q)_m} \dots \overline{\alpha}_{(i_1)_1 \dots (i_q)_1} \\
           & \le 2M^{mq} + M^m K_1 \overline{\gamma}^m,
          \end{aligned}
          \end{equation*}
          where $ \overline{\gamma}= \overset{M_1,\dots,M_q}{\underset{i_{1},\dots,i_{q}=1}{\sum}} \overline{\alpha}_{i_{1},\dots, i_{q}}.$ Moreover, one has
          \begin{align*}
                    \overline{\lim}_{\delta \rightarrow 0}\frac{ \log\Big( N_{\delta}(\mathcal{G}(f^{\alpha}))\Big)}{- \log (\delta)}
                     & \le \overline{\lim}_{m \rightarrow \infty}\frac{ \log\Big(2M^{mq} + M^m K_1 \overline{\gamma}^m\Big) }{m \log (M)}\\ & =
                     1+ \frac{\log (\overline{\gamma})}{\log (M)}.
            \end{align*}
     (b)  Assuming that $  \frac{1}{M_{1}^\sigma} < \alpha_{\max} < 1,$ then one get
     
     \begin{align*}
       &Osc_{f^{\alpha}}[D_{{(i_1)_1 \dots (i_1)_m (i_2)_1  \dots (i_2)_m \dots (i_q)_1 \dots(i_q)_m }}]\\
       \le~ & \alpha_{\max}^m Osc_{f^{\alpha}}[I^q] +
       \frac{2^{\sigma/2} K^* \alpha_{\max}} {M_{1}^\sigma} \left( 1+ \frac{1}{ \alpha_{\max}M_{1}^{ \sigma}} +\dots + \frac{1 }{\alpha_{\max} ^{m-1}M_{1}^{(m-1) \sigma}}\right)\\
       & +2^{\sigma/2} K_s \alpha_{\max}^m \left( 1+ \frac{1}{ \alpha_{\max}M_{1}^{ \sigma}} +\dots + \frac{1 }{\alpha_{\max} ^{m-1}M_{1}^{(m-1) \sigma}}\right) \\ 
       \le~ & \alpha_{\max}^m \left(Osc_{f^{\alpha}}[I^q]+ \left(\frac{2^{\sigma/2} K^* }{M_{1}^\sigma \alpha_{\max}}+2^{\sigma/2} K_s \right)\left(\frac{1}{1 - \frac{1}{\alpha_{\max}M_{1}^{ \sigma}}} \right)\right).
     \end{align*}
     Whence, 
     \[Osc_{f^{\alpha}}[D_{{(i_1)_1 \dots (i_1)_m (i_2)_1  \dots (i_2)_m \dots (i_q)_1 \dots(i_q)_m }}] \le K_2 \alpha_{\max}^m, \]
     where $K_2= Osc_{f^{\alpha}}[I^q]+ \left( \frac{2^{\sigma/2} K^* }{M_{1}^\sigma \alpha_{\max}}+2^{\sigma/2} K_s \right)\left( \frac{1}{1 -  \frac{1}{\alpha_{\max}M_{1}^{ \sigma}}}\right).$
     On similar lines to that in (a), we deduce
     \begin{equation*}
     \begin{aligned}
        \overline{\lim}_{\delta \rightarrow 0}\frac{ \log\Big( N_{\delta}(\mathcal{G}(f^{\alpha}))\Big)}{- \log (\delta)}
       & \le \overline{\lim}_{m \rightarrow \infty}\frac{ \log\Big(2M^{mq} + M^{m(q+1)}M_2 \alpha_{\max}^m\Big) }{m \log (M)}\\ & =
        q+1+ \frac{\log (\alpha_{\max})}{\log (M)}.
     \end{aligned}
     \end{equation*}
     (c) If $ \alpha_{\max} \le  \frac{1}{M_{1}^\sigma},$ then we have
     \[Osc_{f^{\alpha}}[D_{{(i_1)_1 \dots (i_1)_m (i_2)_1  \dots (i_2)_m \dots (i_q)_1 \dots(i_q)_m }}] \le \frac{Osc_{f^{\alpha}}[I^q]+ m 2^{\sigma/2} (K^* + K_s)}{M_{1}^{m\sigma}}.\]
     We get the following in the same way as in the two preceding instances:
     \begin{equation*}
          \begin{aligned}
             & \overline{\lim}_{\delta \rightarrow 0}\frac{ \log\Big( N_{\delta}(\mathcal{G}(f^{\alpha}))\Big)}{- \log (\delta)}\\
             \le& \overline{\lim}_{m \rightarrow \infty}\frac{ \log\Big(2M^{mq} + M^{mq}M_{1}^{(1- \sigma)m} \big(Osc_{f^{\alpha}}[I^q]+ m 2^{\sigma/2} (K^* + K_s)\big)\Big) }{m \log (M)}=
             q+1- \sigma.
          \end{aligned}
          \end{equation*}
    \end{proof}
    \begin{remark}
     The above theorem can be compared with \cite{Maly}. It should be noted that Malysz gives an exact value of box dimension for a special class of multivariate FIFs. However, our result gives bounds of box dimension for a more general class of FIFs.
    \end{remark}
\begin{theorem}\label{BBVL3}
    Let $f , \alpha$ and $s$ be  H\"older continuous with exponent $\sigma$ and $s(x_{1,i_1},\dots,x_{q.i_q})=f(x_{1,i_1},\dots,x_{q.i_q})$ for all $(i_1, \dots, i_q) \in \overset{q}{\underset{k=1}{ \prod}} \partial  \Sigma_{M_{k},0} .$ If $a:= \underset{{(i_1, \dots, i_q)\in \overset{q}{\underset{k=1}{ \prod}}  \Sigma_{M_{k}}}}{\min}\{|a_i|,|c_j| \}$ and\\
    $\max \left\{  \frac{\|\alpha\|_{\infty}}{a^{\sigma}}, \|\alpha\| \right \}< 1 $, then $f^{\alpha}$ is H\"older continuous with exponent $\sigma.$\\
     Moreover, $q \le \dim_H \big(\mathcal{G}(f^{\alpha})\big) \le \underline{\dim}_B \big(\mathcal{G}(f^{\alpha})\big) \le \overline{\dim}_B\big(\mathcal{G}(f^{\alpha})\big) \le q+1- \sigma.$
 \end{theorem}
 \begin{proof}
     Let $ \mathcal{H}^{\sigma}_f(I^q)=\left \{ g \in \mathcal{H}^{\sigma}(I^q): g(x_{1,i_1},\dots,x_{q,i_q})=f(x_{1,i_1},\dots,x_{q.i_q}) \text{ for all } i_k \in \partial  \Sigma_{M_{k},0},~ k\in \Sigma_q\right\}.$ We note that the space $\mathcal{H}^{\sigma}_f(I ^q)$ is a closed subset of $\mathcal{H}^{\sigma}(I^q)$ and consequently  $\mathcal{H}^{\sigma}_f(I ^q)$ is a complete metric space with respect to the metric induced by norm $\|.\|.$ Define a function $T: \mathcal{H}^{\sigma}_f(I ^q) \rightarrow \mathcal{H}^{\sigma}_f(I^q)$ by 
     \[ (T_fg)(x_1,\ldots, x_q)=f(x_1,\ldots, x_q)+\alpha(x_1,\ldots, x_q) ~(g-s)(u_{1,i_1}^{-1}(x_1),\ldots,u_{q,i_q}^{-1}(x_q))\]
     for all $(x_1,\ldots, x_q) \in \overset{q}{\underset{k=1}{ \prod}}I_{k,i_k},$ where $(i_1, \dots, i_q)  \in \overset{q}{\underset{k=1}{ \prod}} \Sigma_{M_{k}} .$
     We begin by proving that $T$ is well-defined. Note that
\begin{align*}
&[T_fg]_\sigma\\
= &\max_{(i_1, \dots, i_q)\in \overset{q}{\underset{k=1}{ \prod}} \Sigma_{M_{k}}} \underset{(x_1,\ldots, x_q), (y_1,\ldots, y_q) \in \overset{q}{\underset{k=1}{ \prod}} I_{k, i_k}}{\sup_{(x_1,\ldots, x_q)\ne (y_1,\ldots, y_q)}} \frac{|T_fg(x_1,\ldots, x_q)-T_fg(y_1,\ldots, y_q)|}{\|(x_1,\ldots, x_q)-(y_1,\ldots, y_q)\|_2^{\sigma}}\\
\le&  \max_{(i_1, \dots, i_q) \in \overset{q}{\underset{k=1}{ \prod}} \Sigma_{M_{k}}} \left[ \underset{(x_1,\ldots, x_q), (y_1,\ldots, y_q) \in \overset{q}{\underset{k=1}{ \prod}} I_{k, i_k}}{\sup_{(x_1,\ldots, x_q)\ne (y_1,\ldots, y_q)}} \frac{|f(x_1,\ldots, x_q)-f(y_1,\ldots, y_q)|}{\|(x_1,\ldots, x_q)-(y_1,\ldots, y_q)\|_2^{\sigma}}\right.\\
& + \underset{(x_1,\ldots, x_q), (y_1,\ldots, y_q) \in \overset{q}{\underset{k=1}{ \prod}} I_{k, i_k}}{\sup_{(x_1,\ldots, x_q)\ne (y_1,\ldots, y_q)}} \frac{\|\alpha\|_{\infty} \Big|(g-s)(u_{1,i_1}^{-1}(x_1),\ldots,u_{q,i_q}^{-1}(x_q))-(g-s)(u_{1,i_1}^{-1}(y_1),\ldots,u_{q,i_q}^{-1}(y_q))\Big|}{\|(x_1,\ldots, x_q)-(y_1,\ldots, y_q)\|_2^{\sigma}}\\
&\left.+\underset{(x_1,\ldots, x_q), (y_1,\ldots, y_q) \in \overset{q}{\underset{k=1}{ \prod}} I_{k, i_k}}{\sup_{(x_1,\ldots, x_q)\ne (y_1,\ldots, y_q)}} \frac{\|g-s\|_{\infty} \Big|\alpha(x_1,\ldots, x_q)-\alpha(y_1,\ldots, y_q)\Big|}{\|(x_1,\ldots, x_q)-(y_1,\ldots, y_q)\|_2^{\sigma}}\right]\\
\le & ~[f]_{\sigma}+ \frac{\|\alpha\|_{\infty}}{a^{\sigma}} \big( [g]_{\sigma}+[s]_{\sigma} \big)+ \|g-s\|_{\infty}[\alpha]_{\sigma}.
\end{align*}
For $g, h \in \mathcal{H}^{\sigma}_f(I^q)$
     \begin{equation*}
          \begin{aligned}
                      \|T_fg -T_fh\| &= \|T_fg -T_fh\|_{\infty} + [T_fg-T_fh]_{\sigma}\\
                      &\le \| \alpha \|_{\infty} \|g -h\|_{\infty} + \frac{|\alpha|}{a^{\sigma}} [g-h]_{\sigma}+\|g-h\|_{\infty}[\alpha]_{\sigma}\\
                      &= \frac{\|\alpha\|_{\infty}}{a^{\sigma}} [g-h]_{\sigma}+\|g-h\|_{\infty}\|\alpha\| \\
             &\le \frac{\|\alpha\|_{\infty}}{a^{\sigma}} \| g-h\|+\|g-h\|_{\infty}\|\alpha\|\\
              &\le  \max \left\{ \frac{\|\alpha\|_{\infty}}{a^{\sigma}}, \|\alpha\| \right \} \|g-h\|    .
          \end{aligned}
          \end{equation*}
          It follows from the hypothesis that $T$ is a contraction function on $ \mathcal{H}^{\sigma}_f(I ^q).$ Hence, a unique fixed point of $T$ is obtained by the Banach contraction principle. In particular, $f^{\alpha} \in \mathcal{H}^{\sigma}_f(I ^q)$.

  Let $K_{f^{\alpha}}$ be the H\"{o}lder constant for a  H\"{o}lder continuous function $f^{\alpha}: I^q \rightarrow \mathbb{R}$ with exponent $\sigma$. For $\delta= \frac{1}{M^{m}}$ and using the Lemma \ref{BBVL2} we get
\begin{align*}
\overline{\dim}_B\big(\mathcal{G}(f^{\alpha})\big)=
\varlimsup_{\delta \rightarrow 0} \frac{\log \left(N_{\delta}\left(\operatorname{\mathcal{G}}\left(f^{\alpha}\right)\right)\right)}{-\log (\delta)} & \leq \varlimsup_{m \rightarrow \infty} \frac{\log \left(2 M^{ mq}+M^{m(1-\sigma)} M^{mq} K_{f^{\alpha}} 2^{\frac{\sigma}{2}}\right)}{m \log (M)} =1+q-\sigma,
\end{align*}
and this completes the proof.
\end{proof}
\begin{remark}
 The above theorem can be compared with Theorem of \cite{Jha}, wherein Jha and Verma have shown a similar result for univariate fractal functions. 
\end{remark}

\begin{example}
Consider the following function in the domain $[-1,1]\times[-1,1]$ as the germ function
\[
f(x,y)=41(y^3-x^5)^2 + (y-x^2)^3
\]
 Let us consider the following parameters:
\begin{enumerate}[(i)]
    \item A net $\Delta$ determined by the partition $\{-1,-0.5,0,0.5,1\}$ of $[-1,1]$.
    \item Scaling function $\alpha(x,y)=0.1, 0.2, 0.4, 0.5, 0.7, 0.9, $ respectively for all $(x,y)\in [-1,1]\times [-1,1].$
    \item Base function $s(x,y)=x^3y^5f(x,y).$
    \end{enumerate}
\begin{figure}[H]
\begin{minipage}{0.3\textwidth}
\includegraphics[width=1.0\linewidth]{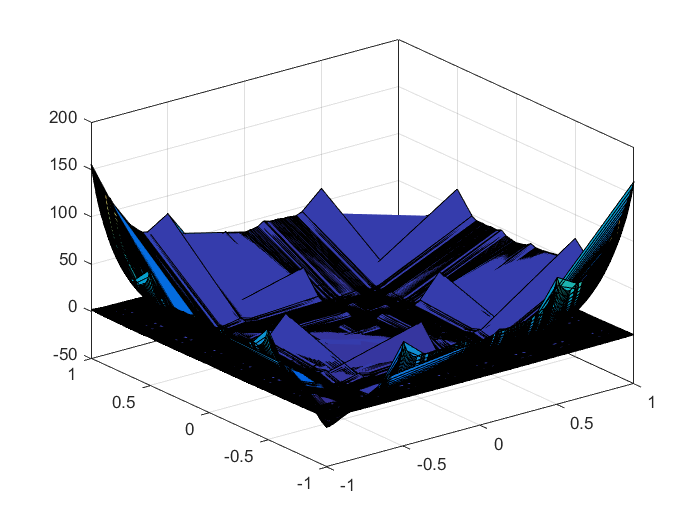}
{\hspace*{1.5cm} $\alpha(x,y)=0.1$}
\end{minipage}\hspace*{0.3cm}
\begin{minipage}{0.3\textwidth}
\includegraphics[width=1.0\linewidth]{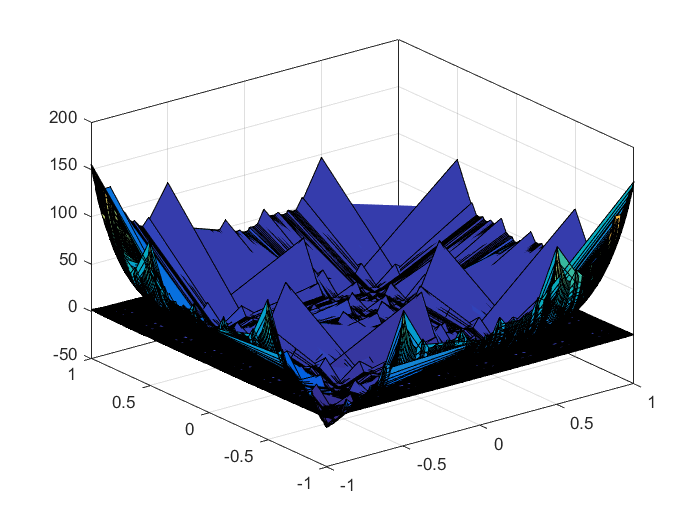}
{\hspace*{1.5cm} $\alpha(x,y)=0.2$}
\end{minipage}\hspace*{0.3cm}
\begin{minipage}{0.3\textwidth}
\includegraphics[width=1.0\linewidth]{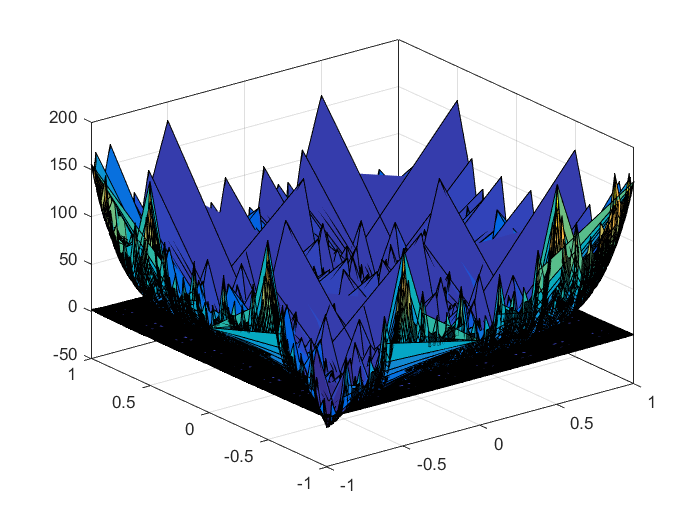}
{\hspace*{1.5cm}}$\alpha(x,y)=0.4$
\end{minipage}
\end{figure}
\begin{figure}[H]
\begin{minipage}{0.3\textwidth}
\includegraphics[width=1.0\linewidth]{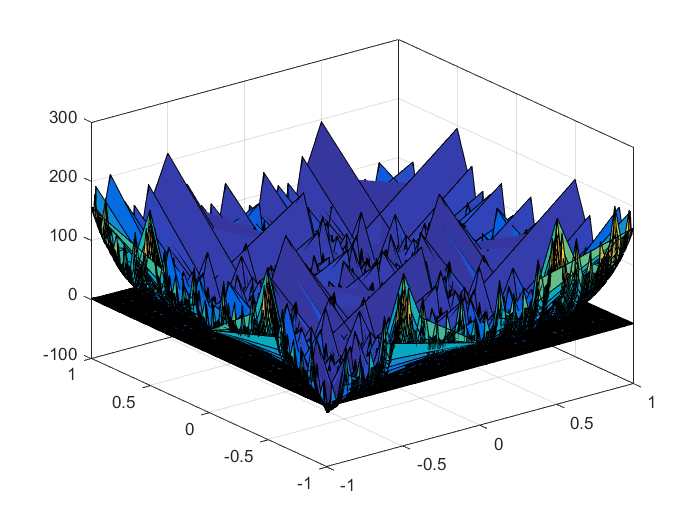}
{\hspace*{1.5cm}  $\alpha(x,y)=0.5$}
\end{minipage}\hspace*{0.3cm}
\begin{minipage}{0.3\textwidth}
\includegraphics[width=1.0\linewidth]{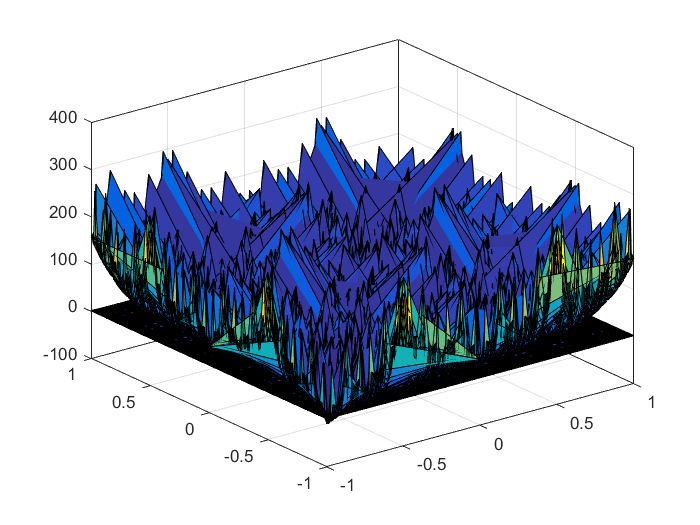}
{\hspace*{1.5cm}$\alpha(x,y)=0.7$}
\end{minipage}\hspace*{0.3cm}
\begin{minipage}{0.3\textwidth}
\includegraphics[width=1.0\linewidth]{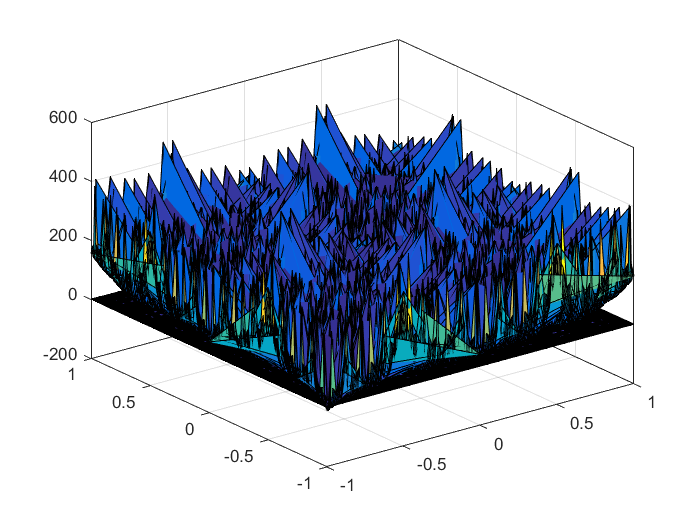}
{\hspace*{1.5cm} $\alpha(x,y)=0.9$}
\end{minipage} 
\caption{$\mathcal{G}(f^{\alpha}) $ on different Scaling function}
\label{fig1}
\end{figure}
\end{example}

\subsection{Restrictions of $\alpha$-fractal function on axis}\label{Restrictions}
Let $f^{\alpha}$ be a fractal interpolation function as in Subsection \ref{FIF}. Now define a function $g^{\alpha}:I_{1,i_1}\rightarrow \mathbb{R}$ such that 
\begin{equation}\label{proj}
g^{\alpha^g}(x_1)=f^{\alpha}\left(x_1,u_{2,1}(0),\ldots,u_{q,1}(0)\right) \text{ for all } x_1\in I_{1,i_1}
\end{equation}
where $\alpha^g(x_1)=\alpha \left(x_1,u_{2,1}(0),\ldots,u_{q,1}(0)\right),~ s^g(x_1,\ldots, x_q)=s(x_1,u_{2,1}(0),\ldots, u_{q,1}(0))$.

\begin{theorem}
Let $f\in \mathcal{C}(I^q)$ be a continuous function and $f^{\alpha}$ the $\alpha$-fractal function corresponding to $f$, then $g^{\alpha^g}$ as defined in Equation \eqref{proj} will be the induced $\alpha$-fractal function corresponding to a continuous function $g:I\rightarrow \mathbb{R}$, where $g(x^1)=f(x^1,u_{2,1}(0),\ldots, u_{q,1}(0))$ for all $x^1\in I$. Moreover, $\dim\left(\mathcal{G}\left(g^{\alpha^g}\right)\right)\leq \dim \left( \mathcal{G}\left(f^{\alpha}\right)\right)$.
\end{theorem}
\begin{proof}
To prove that $g^{\alpha^g}$ is an $\alpha$-fractal function corresponding to $g$, it is sufficient to prove that $g^{\alpha^g}$ satisfies the self-referential equation. With the definition of $u_{k,i_k}$, it is clear that for all $k\in  \Sigma_{q}$, $u_{k,1}(0)=0$, then we have
\begin{align*}
    &g^{\alpha^g}(x_1)\\
    =& f^{\alpha}\left(x_1,u_{2,1}(0),\ldots, u_{q,1}(0)\right)\\
    =&f(x_1,u_{2,1}(0),\ldots, u_{q,1}(0))+\alpha(x_1,u_{1,1}(0),\ldots, u_{q,1}(0))\left[f^{\alpha}\left(u_{1,i_1}^{-1}(x_1),0,\ldots, 0\right)-s\left(u_{1,i_1}^{-1}(x_1),0,\ldots, 0\right)\right]\\
    =&g(x_1)+\alpha^{g}(x_1)\left[g^{\alpha^g}\left(u_{1,i_1}^{-1}(x_1)\right)-s^{\alpha^g}\left(u_{1,i_1}^{-1}(x_1))\right)\right.]
\end{align*}
It is satisfying the self referential Equation \eqref{alphafrac}, hence $g^{\alpha^{g}}$ is an $\alpha$-fractal function corresponding to $g$.\\ 
Moreover, $\mathcal{G}\left(g^{\alpha^g}\right)\subseteq \mathcal{G}\left(f^{\alpha}\right)$, implies that 
\[\dim\left(\mathcal{G}\left(g^{\alpha^g}\right)\right)\leq \dim \left( \mathcal{G}\left(f^{\alpha}\right)\right).\]
This completes the proof.
\end{proof}

\begin{example}
Let $f: I^2 \to \mathbb{R}$ be a real-valued function. Using $f$, we define $g: I \to \mathbb{R}$ as follows:
\begin{align}\label{fg}
    g(x)=f(x,0)
\end{align}

It can be seen that 
\[\dim\mathcal{G}(g) \le \dim \mathcal{G} (f).\]
Consider $f(x,y)=x^2+y^2$, then the function $g(x)=x^2$ is obtained by using Equation \ref{fg}. It can be easily seen that $\dim \mathcal{G}(f) =2$ and $\dim \mathcal{G}(g) =1$, that is, $\dim \mathcal{G} (g) \le \dim \mathcal{G}(f)$. Now, take $f(x,y)= W(x)+y$, where $W: I \to \mathbb{R}$ is the Weierstrass function. Similarly, we get $g(x)= W(x)$. Finally,  $\dim \mathcal{G}(f)= 1+ \dim(W)$ and $\dim \mathcal{G}(g)= \dim \mathcal{G}(W)$. Hence, $\dim \mathcal{G}(g) \le \dim \mathcal{G}(f)$. Now, select $f(x,y)=W(x)y$, then $g(x)=0$. Thus, we get $\dim \mathcal{G}(fg)=1$.

\end{example}
\begin{theorem}\label{BBVL71}
    Let $f , \alpha$ and $s$ be  H\"older continuous with exponent $\sigma$ and $s(x_{1,i_1},\dots,x_{q.i_q})=f(x_{1,i_1},\dots,x_{q.i_q})$ for all $(i_1, \dots, i_q) \in \overset{q}{\underset{k=1}{ \prod}} \partial  \Sigma_{M_{k},0} .$ If $a:= \underset{{(i_1, \dots, i_q)\in \overset{q}{\underset{k=1}{ \prod}}  \Sigma_{M_{k}}}}{\min}\{|a_i|,|c_j| \}$ and\\
    $  \max \left\{  \frac{\|\alpha\|_{\infty}}{a^{\sigma}}, \|\alpha\| \right \}< 1 $, then $g^{\alpha^g}$ is H\"older continuous with exponent $\sigma.$ \\
    Moreover, $1 \le \dim_H \big(\mathcal{G}(g^{\alpha^g})\big) \le \underline{\dim}_B \big(\mathcal{G}(g^{\alpha^g})\big) \le \overline{\dim}_B\big(\mathcal{G}(g^{\alpha^g})\big) \le 2- \sigma.$
 \end{theorem}
 \begin{proof}
 In view of Theorem \ref{BBVL3}, the fractal function $f^{\alpha}$ is H\"older continuous with exponent $\sigma.$ Since \begin{equation*}
g^{\alpha^g}(x_1)=f^{\alpha}\left(x_1,u_{2,1}(0),\ldots,u_{q,1}(0)\right) \text{ for all } x_1\in I_{1,i_1}
\end{equation*}
where $\alpha^g(x_1)=\alpha \left(x_1,u_{2,1}(0),\ldots,u_{q,1}(0)\right),~ s^g(x_1,\ldots, x_q)=s(x_1,u_{2,1}(0),\ldots, u_{q,1}(0))$, we have 
\begin{equation*}
|g^{\alpha^g}(x)- g^{\alpha^g}(y)| = |f^{\alpha}\left(x,u_{2,1}(0),\ldots,u_{q,1}(0)\right)-f^{\alpha}\left(y,u_{2,1}(0),\ldots,u_{q,1}(0)\right)| \le |x -y|^{\sigma}.
\end{equation*}
That is, $g^{\alpha^g}$ is H\"older continuous with exponent $\sigma.$ Now, by \cite[Corollary $11.2$]{Fal} it is obvious that $ \overline{\dim}_B\big(\mathcal{G}(g^{\alpha^g})\big) \le 2- \sigma.$ This completes the proof.
 \end{proof}
%%%%%%%%%%%%%%%%%%%%%%%%%%%%%%%%%%%%%%%%%%%%%%%%%%%%%%%%%%%%%%%%%%%%%%%%%%%%%%%%%%%%%%%%%%%%%%%%%%%%%%%%%%%%%%%%%%%%%%%%%%%%%%%%%%%%%%%%%%%%%%%%%%%%%%%%%%%%%%%%
\section{Mixed Riemann-Liouville fractional integral}\label{Mrl}
\begin{definition}
 The Mixed Riemann-Liouville fractional integral of a multivariate continuous function $f: I_q \to \mathbb{R}$ is defined as
 \begin{align*}
     &\left(\mathfrak{J}^{(\beta_1,\ldots, \beta_q)}f\right)(x_1,\ldots, x_q)\\
     = &\frac{1}{\Gamma (\beta_1)\cdots \Gamma(\beta_q) } \int_{0} ^{x_1} \cdots \int_{0}^{x_q} (x_1-t_1)^{\beta_1-1} \cdots  (x_q-t_q)^{\beta_q-1} f(t_1,\ldots, t_q)~ \mathrm{d}t_1\ldots \mathrm{d}t_q,
 \end{align*}
     where $\beta_1,\ldots \beta_q > 0$ and $(x_1,\ldots, x_q) \ne (-1,\ldots,-1) \text{ for all } (x_1,\ldots, x_q) \in I^q.$
\end{definition}
\begin{theorem}\label{Thm5.2}
Let $\mathcal{A}$ be the FIF defined as in Equation \eqref{interpolation} which is determined by the IFS defined in Equation \eqref{IFS}, then $\mathfrak{J}^{(\beta_1,\ldots, \beta_q)}\mathcal{A}$ will satisfy the following self referential equation
\begin{equation*}
\mathfrak{J}^{(\beta_1,\ldots, \beta_q)}\mathcal{A}\left(x_1,\ldots,x_q\right)=\reallywidehat{F}_{i_1\cdots i_q}\left(u_{1,i_1}^{-1}(x_1),\ldots,u_{q,i_q}^{-1}(x_q),~ \mathfrak{J}^{(\beta_1,\ldots, \beta_q)}\mathcal{A}\left(u_{1,i_1}^{-1}(x_1),\ldots,u_{q,i_q}^{-1}(x_q))\right)\right),
\end{equation*}
where $\reallywidehat{F}_{i_1\cdots i_q}: X\rightarrow \mathbb{R}$ is defined as,
\[\reallywidehat{F}_{i_1\cdots i_q}(x_1,\ldots,x_q,z)=\delta a_{q,i_{q}}^{\beta_q}\cdots a_{1,i_{1}}^{\beta_1} z+\reallywidehat{\mathcal{B}}_{i_1\cdots i_q}(x_1,\ldots , x_q),\]
where
\begin{align*}
&\reallywidehat{\mathcal{B}}_{i_1\cdots i_q}(x_1,\ldots,x_q)\\
=&~\mathcal{M}\left(\int_{0}^{u_{1,i_{1}}(x_1)}\ldots \int_{0}^{u_{q,i_q}(0)}(u_{1,i_1}(x_1)-t_1)^{\beta_1-1}\cdots (u_{q,i_q}(x_q)-t_q)^{\beta_q-1}\mathcal{A}(t_1,\ldots, t_q)~\mathrm{d}t_1\cdots \mathrm{d}t_q\right.\\
  &+ a_{q,i_{q}}^{\beta_q}\int_{0}^{u_{1,i_1}(x_1)}\cdots \int_{0}^{u_{q-1,i_{q-1}}(0)}\int_{0}^{x_q}(u_{1,i_1}(x_1)-t_1)^{\beta_1-1}\cdots (x_q-w_q)^{\beta_q-1}\\
  & \hspace{8cm}\mathcal{A}(t_1,\ldots, u_{q,i_q}(w_q))~\mathrm{d}t_1\cdots \mathrm{d}t_{q-1}\mathrm{d}w_q\\
  &+ a_{q,i_q}^{\beta_q}a_{q-1,i_{q-1}}^{\beta_{q-1}}\int_{0}^{u_{1,i_1}(x_1)}\cdots \int_{0}^{x_{q-1}}\int_{0}^{x_q}(u_{1,i_1}(x_1)-t_1)^{\beta_1-1}\cdots (x_{q-1}-w_{q-1})^{\beta_{q-1}-1}(x_q-w_q)^{\beta_q-1}\\
  &\hspace{5cm} \mathcal{A}(t_1,\ldots,t_{q-2}, u_{q-1,i_{q-1}}(w_{q-1}),u_{q,i_q}(w_q))~\mathrm{d}t_1\cdots \mathrm{d}t_{q-2}\mathrm{d}w_{q-1}\mathrm{d}w_q\\
  &+\cdots +a_{q,i_{q}}^{\beta_q}\cdots a_{2,i_{2}}^{\beta_2}\int_{0}^{u_{1,i_{1}}(0)}\int_{0}^{x_2}\cdots \int_{0}^{x_q}(u_{1,i_1}(x_1)-t_1)^{\beta_1-1}(x_2-w_2)^{\beta_2-1}\cdots (x_q-w_q)^{\beta_q-1}\\
  &\hspace{7cm}\mathcal{A}(t_1,u_{2,i_2}(w_2),\ldots,u_{q,i_q}(w_q))~\mathrm{d}t_1\mathrm{d}w_2\cdots \mathrm{d}w_{q}\\
  &\left. +a_{q,i_{q}}^{\beta_q}\cdots a_{1,i_{1}}^{\beta_1}\int_{0}^{x_1}\cdots \int_{0}^{x_q}(x_1-w_1)^{\beta_1-1}\cdots (x_q-w_q)^{\beta_q-1}\mathcal{B}_{i_1\cdots i_q}(w_1,\ldots, w_q)~\mathrm{d}w_1\cdots \mathrm{d}w_{q}\right),
\end{align*}
where $\mathcal{M}= \frac{1}{\Gamma (\beta_1)\cdots \Gamma(\beta_q)}$.
\end{theorem}
\begin{proof}
For $(x_1,\ldots, x_q)\in I^q$, we have
\begin{align*}
  &\mathfrak{J}^{(\beta_1,\ldots, \beta_q)}\mathcal{A}\left(u_{1,i_{1}}(x_1),\ldots, u_{q,i_{q}}(x_q)\right)\\
  =& \mathcal{M}\int_{0} ^{u_{1,i_{1}}(x_1)} \cdots \int_{0}^{u_{q,i_{q}}(x_q)} (u_{1,i_{1}}(x_1)-t_1)^{\beta_1-1} \cdots  (u_{q,i_{q}}(x_q)-t_q)^{\beta_q-1}\mathcal{A}(t_1,\ldots, t_q)~ \mathrm{d}t_1\cdots \mathrm{d}t_q\\
  =& \mathcal{M} \left(\int_{0} ^{u_{1,i_{1}}(x_1)} \cdots \int_{0}^{u_{q,i_{q}}(0)} (u_{1,i_{1}}(x_1)-t_1)^{\beta_1-1} \cdots  (u_{q,i_{q}}(x_q)-t_q)^{\beta_q-1}\mathcal{A}(t_1,\ldots, t_q)~ \mathrm{d}t_1\cdots \mathrm{d}t_q\right.\\
  &\hspace{0.5cm} +\left.\int_{0} ^{u_{1,i_{1}}(x_1)} \cdots \int_{u_{q,i_{q}}(0)}^{u_{q,i_{q}}(x_q)} (u_{1,i_{1}}(x_1)-t_1)^{\beta_1-1} \cdots  (u_{q,i_{q}}(x_q)-t_q)^{\beta_q-1}\mathcal{A}(t_1,\ldots, t_q)~ \mathrm{d}t_1\cdots \mathrm{d}t_q \right)\\
  & \text{ put } t_q=u_{q,i_{q}}(w_q) \text{ in the second integral, then we get}\\
  =& \mathcal{M} \left(\int_{0} ^{u_{1,i_{1}}(x_1)} \cdots \int_{0}^{u_{q,i_{q}}(0)} (u_{1,i_{1}}(x_1)-t_1)^{\beta_1-1} \cdots  (u_{q,i_{q}}(x_q)-t_q)^{\beta_q-1}\mathcal{A}(t_1,\ldots, t_q)~ \mathrm{d}t_1\cdots \mathrm{d}t_q\right.
  \end{align*}
  \begin{align*}
  &+a_{q,i_q}\int_{0} ^{u_{1,i_{1}}(x_1)} \cdots \int_{0}^{x_q} (u_{1,i_{1}}(x_1)-t_1)^{\beta_1-1} \cdots  (u_{q,i_{q}}(x_q)-u_{q,i_{q}}(w_q))^{\beta_q-1}\\
  &\hspace{9cm} \mathcal{A}(t_1,\ldots, u_{q,i_{q}}(w_q))~ \mathrm{d}t_1\cdots \mathrm{d}t_{q-1}\mathrm{d}w_q \Bigg)\\
  =&\mathcal{M}\left(\int_{0}^{u_{1,i_{1}}(x_1)}\ldots \int_{0}^{u_{q,i_q}(0)}(u_{1,i_1})(x_1)-t_1)^{\beta_1-1}\cdots (u_{q,i_q}(x_q)-t_q)^{\beta_q-1}\mathcal{A}(t_1,\ldots, t_q)~\mathrm{d}t_1\cdots \mathrm{d}t_q\right.\\
  &+ a_{q,i_{q}}^{\beta_q}\int_{0}^{u_{1,i_1}(x_1)}\cdots \int_{0}^{u_{q-1,i_{q-1}}(0)}\int_{0}^{x_q}(u_{1,i_1}(x_1)-t_1)^{\beta_1-1}\cdots (x_q-w_q)^{\beta_q-1}\\
  &\hspace{9cm}\mathcal{A}(t_1,\ldots, u_{q,i_q}(w_q))~\mathrm{d}t_1\cdots \mathrm{d}t_{q-1}\mathrm{d}w_q\\
  &+a_{q,i_{q}}^{\beta_q}\int_{0}^{u_{1,i_1}(x_1)}\cdots \int_{u_{q-1,i_{q-1}}(0)}^{u_{q-1,i_{q-1}}(x_{q-1})}\int_{0}^{x_q}(u_{1,i_1}(x_1)-t_1)^{\beta_1-1}\cdots (x_q-w_q)^{\beta_q-1}\\
  &\hspace{9cm}\mathcal{A}(t_1,\ldots, u_{q,i_q}(w_q))~\mathrm{d}t_1\cdots \mathrm{d}t_{q-1}\mathrm{d}w_q\Bigg)\\
  & \text{ put } t_{q-1}=u_{q-1,i_{q-1}}(w_{q-1}) \text{ in the third integral, then}\\
  =&\mathcal{M}\left(\int_{0}^{u_{1,i_{1}}(x_1)}\ldots \int_{0}^{u_{q,i_q}(0)}(u_{1,i_1})(x_1)-t_1)^{\beta_1-1}\cdots (u_{q,i_q}(x_q)-t_q)^{\beta_q-1}\mathcal{A}(t_1,\ldots, t_q)~\mathrm{d}t_1\cdots \mathrm{d}t_q\right.\\
  &+ a_{q,i_{q}}^{\beta_q}\int_{0}^{u_{1,i_1}(x_1)}\cdots \int_{0}^{u_{q-1,i_{q-1}}(0)}\int_{0}^{x_q}(u_{1,i_1}(x_1)-t_1)^{\beta_1-1}\cdots (x_q-w_q)^{\beta_q-1}\\
  &\hspace{9cm}\mathcal{A}(t_1,\ldots, u_{q,i_q}(w_q))~\mathrm{d}t_1\cdots \mathrm{d}t_{q-1}\mathrm{d}w_q\\
  &+a_{q,i_q}^{\beta_q}a_{q-1,i_{q-1}}\int_{0}^{u_{1,i_1}(x_1)}\cdots \int_{0}^{x_{q-1}}\int_{0}^{x_q}(u_{1,i_1}(x_1)-t_1)^{\beta_1-1}\cdots \left(u_{q-1,i_{q-1}}(x_{q-1})\right.\\
  &\left.-u_{q-1,i_{q-1}}(w_{q-1})\right)^{\beta_{q-1}-1}
  (x_q-w_q)^{\beta_q-1} \mathcal{A}(t_1,\ldots,t_{q-2}, u_{q-1,i_{q-1}}(w_{q-1}),u_{q,i_q}(w_q))\\
  &\hspace{11cm}\mathrm{d}t_1\cdots \mathrm{d}t_{q-2}\mathrm{d}w_{q-1}\mathrm{d}w_q \Bigg)\\
  =&\mathcal{M}\left(\int_{0}^{u_{1,i_{1}}(x_1)}\ldots \int_{0}^{u_{q,i_q}(0)}(u_{1,i_1})(x_1)-t_1)^{\beta_1-1}\cdots (u_{q,i_q}(x_q)-t_q)^{\beta_q-1}\mathcal{A}(t_1,\ldots, t_q)~\mathrm{d}t_1\cdots \mathrm{d}t_q\right.\\
  &+ a_{q,i_{q}}^{\beta_q}\int_{0}^{u_{1,i_1}(x_1)}\cdots \int_{0}^{u_{q-1,i_{q-1}}(0)}\int_{0}^{x_q}(u_{1,i_1}(x_1)-t_1)^{\beta_1-1}\cdots (x_q-w_q)^{\beta_q-1}\mathcal{A}(t_1,\ldots, u_{q,i_q}(w_q))\\
  &\hspace{13cm}\mathrm{d}t_1\cdots \mathrm{d}t_{q-1}\mathrm{d}w_q\\
  &+ a_{q,i_q}^{\beta_q}a_{q-1,i_{q-1}}^{\beta_{q-1}}\int_{0}^{u_{1,i_1}(x_1)}\cdots \int_{0}^{x_{q-1}}\int_{0}^{x_q}(u_{1,i_1}(x_1)-t_1)^{\beta_1-1}\cdots (x_{q-1}-w_{q-1})^{\beta_{q-1}-1}(x_q-w_q)^{\beta_q-1}\\
  &\hspace{5cm} \mathcal{A}(t_1,\ldots,t_{q-2}, u_{q-1,i_{q-1}}(w_{q-1}),u_{q,i_q}(w_q))~\mathrm{d}t_1\cdots \mathrm{d}t_{q-2}\mathrm{d}w_{q-1}\mathrm{d}w_q \Bigg)\\
   &\text{going on with similar steps, we get}\\
  =&\mathcal{M}\left(\int_{0}^{u_{1,i_{1}}(x_1)}\ldots \int_{0}^{u_{q,i_q}(0)}(u_{1,i_1})(x_1)-t_1)^{\beta_1-1}\cdots (u_{q,i_q}(x_q)-t_q)^{\beta_q-1}\mathcal{A}(t_1,\ldots, t_q)~\mathrm{d}t_1\cdots \mathrm{d}t_q\right.\\
  &+a_{q,i_{q}}^{\beta_q}\int_{0}^{u_{1,i_1}(x_1)}\cdots \int_{0}^{u_{q-1,i_{q-1}}(0)}\int_{0}^{x_q}(u_{1,i_1}(x_1)-t_1)^{\beta_1-1}\cdots (x_q-w_q)^{\beta_q-1}\mathcal{A}(t_1,\ldots, u_{q,i_q}(w_q))\\
  &\hspace{12cm}\mathrm{d}t_1\cdots \mathrm{d}t_{q-1}\mathrm{d}w_q\\
  &+ a_{q,i_q}^{\beta_q}a_{q-1,i_{q-1}}^{\beta_{q-1}}\int_{0}^{u_{1,i_1}(x_1)}\cdots \int_{0}^{x_{q-1}}\int_{0}^{x_q}(u_{1,i_1}(x_1)-t_1)^{\beta_1-1}\cdots (x_{q-1}-w_{q-1})^{\beta_{q-1}-1}(x_q-w_q)^{\beta_q-1}\\
  &\hspace{5cm} \mathcal{A}(t_1,\ldots,t_{q-2}, u_{q-1,i_{q-1}}(w_{q-1}),u_{q,i_q}(w_q))~\mathrm{d}t_1\cdots \mathrm{d}t_{q-2}\mathrm{d}w_{q-1}\mathrm{d}w_q
  \end{align*}
  \begin{align*}
  &+\cdots +a_{q,i_{q}}^{\beta_q}\cdots a_{2,i_{2}}^{\beta_2}\int_{0}^{u_{1,i_{1}}(0)}\int_{0}^{x_2}\cdots \int_{0}^{x_q}(u_{1,i_1}(x_1)-t_1)^{\beta_1-1}(x_2-w_2)^{\beta_2-1}\cdots (x_q-w_q)^{\beta_q-1}\\
  &\hspace{8cm}\mathcal{A}(t_1,u_{2,i_2}(w_2),\ldots,u_{q,i_q}(w_q))~\mathrm{d}t_1\mathrm{d}w_2\cdots \mathrm{d}w_{q}\\
  &+\left. a_{q,i_{q}}^{\beta_q}\cdots a_{1,i_{1}}^{\beta_1}\int_{0}^{x_1}\cdots \int_{0}^{x_q}(x_1-w_1)^{\beta_1-1}\cdots (x_q-w_q)^{\beta_q-1}\mathcal{A}(u_{1,i_1}(w_1),\ldots,u_{q,i_q}(w_q))~\mathrm{d}w_1\cdots \mathrm{d}w_{q}\right)\\
  =&\mathcal{M}\left(\int_{0}^{u_{1,i_{1}}(x_1)}\ldots \int_{0}^{u_{q,i_q}(0)}(u_{1,i_1})(x_1)-t_1)^{\beta_1-1}\cdots (u_{q,i_q}(x_q)-t_q)^{\beta_q-1}\mathcal{A}(t_1,\ldots, t_q)~\mathrm{d}t_1\cdots \mathrm{d}t_q\right.\\
  &+ a_{q,i_{q}}^{\beta_q}\int_{0}^{u_{1,i_1}(x_1)}\cdots \int_{0}^{u_{q-1,i_{q-1}}(0)}\int_{0}^{x_q}(u_{1,i_1}(x_1)-t_1)^{\beta_1-1}\cdots (x_q-w_q)^{\beta_q-1}\mathcal{A}(t_1,\ldots, u_{q,i_q}(w_q))\\
  &\hspace{12cm}\mathrm{d}t_1\cdots \mathrm{d}t_{q-1}\mathrm{d}w_q\\
  &+ a_{q,i_q}^{\beta_q}a_{q-1,i_{q-1}}^{\beta_{q-1}}\int_{0}^{u_{1,i_1}(x_1)}\cdots \int_{0}^{x_{q-1}}\int_{0}^{x_q}(u_{1,i_1}(x_1)-t_1)^{\beta_1-1}\cdots (x_{q-1}-w_{q-1})^{\beta_{q-1}-1}(x_q-w_q)^{\beta_q-1}\\
  &\hspace{5cm} \mathcal{A}(t_1,\ldots,t_{q-2}, u_{q-1,i_{q-1}}(w_{q-1}),u_{q,i_q}(w_q))~\mathrm{d}t_1\cdots \mathrm{d}t_{q-2}\mathrm{d}w_{q-1}\mathrm{d}w_q\\
  &+\cdots +a_{q,i_{q}}^{\beta_q}\cdots a_{2,i_{2}}^{\beta_2}\int_{0}^{u_{1,i_{1}}(0)}\int_{0}^{x_2}\cdots \int_{0}^{x_q}(u_{1,i_1}(x_1)-t_1)^{\beta_1-1}(x_2-w_2)^{\beta_2-1}\cdots (x_q-w_q)^{\beta_q-1}\\
  &\hspace{7cm}\mathcal{A}(t_1,u_{2,i_2}(w_2),\ldots,u_{q,i_q}(w_q))~\mathrm{d}t_1\mathrm{d}w_2\cdots \mathrm{d}w_{q}\\
  &+a_{q,i_{q}}^{\beta_q}\cdots a_{1,i_{1}}^{\beta_1}\int_{0}^{x_1}\cdots \int_{0}^{x_q}(x_1-w_1)^{\beta_1-1}\cdots (x_q-w_q)^{\beta_q-1}\delta\mathcal{A}(w_1,\ldots, w_q)~\mathrm{d}w_1\cdots \mathrm{d}w_{q}\\
  &\left. +a_{q,i_{q}}^{\beta_q}\cdots a_{1,i_{1}}^{\beta_1}\int_{0}^{x_1}\cdots \int_{0}^{x_q}(x_1-w_1)^{\beta_1-1}\cdots (x_q-w_q)^{\beta_q-1}\mathcal{B}_{i_1\cdots i_q}(w_1,\ldots, w_q)~\mathrm{d}w_1\cdots \mathrm{d}w_{q}\right)\\
  =&\delta a_{q,i_{q}}^{\beta_q}\cdots a_{1,i_{1}}^{\beta_1} \mathfrak{J}^{(\beta_1,\ldots, \beta_q)}\mathcal{A}\left(x_1,\ldots, x_q\right)+\reallywidehat{\mathcal{B}}_{i_1\cdots i_q}(x_1,\ldots, x_q) ,
 \end{align*}
 establishing the claim. 
 \end{proof}

\begin{remark}
 The above theorem can be compared with Theorem $5.3$ of \cite{SS}, wherein the calculus of fractal surfaces has been studied. 
\end{remark}
\section*{Conclusion and Future Problems}
In this paper, we have given the construction of multivariate fractal interpolation function corresponding to the given data set (Subsection \ref{subsec1}). We have noticed that the constructed fractal interpolation function is an attractor of an IFS (Theorem \ref{matching}) and satisfies the self-referential equation (Equation \eqref{interpolation}). Next, we have established the existence of multivariate $\alpha$-fractal function corresponding to the continuous function defined on $I^q$ (Subsection \ref{FIF}). Next, we have studied the fractal dimension of multivariate $\alpha$-fractal function (Theorem \ref{Th4.10}, Theorem \ref{BBVL3}), and also studied the restriction of $\alpha$-fractal function on the co-ordinate axis (Subsection \ref{Restrictions}). In the end, we have proved that the mixed Riemann-Liouville fractional integral of a multivariate fractal interpolation function satisfies the self-referential equation (Theorem \ref{Thm5.2}). 

In Subsection \ref{Restrictions}, we have studied the restriction of multivariate $\alpha$-fractal functions on co-ordinate axes. Next we may try to explore the projections of multivariate $\alpha$-fractal functions and find the fractal dimension of such projections. A glimpse of study of orthogonal projections of fractal structures can be seen in \cite{Fal}.

\section*{Author's Contribution}
All authors contributed equally.

\section*{Acknowledgements}
 This work is supported by MHRD Fellowship to the 1st and 2nd author as TA-ship at the Indian Institute of Technology (BHU), Varanasi.

\bibliographystyle{abbrv}
\bibliography{Arxiv}

\end{document}